\documentclass[10pt, leqno]{article}
\usepackage{geometry}           
\usepackage{amsmath, amsthm, amssymb, amscd, physics, enumitem, titlesec, titletoc, accents}
\usepackage[mathscr]{euscript}
\usepackage{marvosym}

\renewcommand{\d}{\partial}
\newcommand{\R}{\mathbb{R}}

\newcommand{\Ric}{\mathrm{Ric}}

\makeatletter
\newcommand\incircbin
{%
  \mathpalette\@incircbin
}
\newcommand\@incircbin[2]
{%
  \mathbin%
  {%
    \ooalign{\hidewidth$#1#2$\hidewidth\crcr$#1\bigcirc$}%
  }%
}

\makeatother

\newtheorem{theorem}{Theorem}[section]
\newtheorem{corollary}[theorem]{Corollary}
\newtheorem{lemma}[theorem]{Lemma}
\newtheorem{proposition}[theorem]{Proposition}
\theoremstyle{definition}
\newtheorem{definition}[theorem]{Definition}

\numberwithin{equation}{section}

\titleformat{\section}[block]{\scshape\filcenter}{\thesection.}{3pt}{}
\titleformat{\subsection}[block]{\scshape\filcenter}{\thesubsection.}{3pt}{}
\titlecontents{section}[0pt]{}{\thecontentslabel. \,}{}{\hfill\contentspage}

\title{\large{\textbf{SINGULARITY MODELS OF PINCHED SOLUTIONS OF MEAN CURVATURE FLOW IN HIGHER CODIMENSION}}}
\author{\textsc{\small KEATON NAFF}}
\date{}

\begin{document}
\maketitle

\abstract{We consider ancient solutions to the mean curvature flow in $\mathbb{R}^{n+1}$ ($n \geq 3$) that are weakly convex, uniformly two-convex, and satisfy derivative estimates $|\nabla A| \leq \gamma_1 |H|^2, |\nabla^2 A| \leq \gamma_2 |H|^3$. We show that such solutions are noncollapsed. As an application, in arbitrary codimension, we consider compact $n$-dimensional ($n \geq 5$) solutions to the mean curvature flow in $\mathbb{R}^N$ that satisfy the pinching condition $|H| > 0$ and $|A|^2 < c(n) |H|^2$, $c(n) = \min\{\frac{1}{n-2}, \frac{3(n+1)}{2n(n+2)}\}$. We conclude that any blow-up model at the first singular time must be a codimension one shrinking sphere, shrinking cylinder, or translating bowl soliton.}

\section{Introduction}

In this paper, we prove a noncollapsing result for ancient, weakly convex, uniformly two-convex solutions of the mean curvature satisfying two derivative estimates. The uniqueness of ancient solutions that are uniformly two-convex and noncollapsed (in the sense of Sheng and Wang \cite{SW09}) has been studied in the noncompact case by Brendle and Choi \cite{BC18, BC19} and in the compact case by Angenent, Daskalopoulos, and Sesum \cite{ADS18, ADS19}. For applications to higher codimension, we replace the noncollapsed assumption with two derivative estimates. Our main theorem is:

\begin{theorem}\label{main}
Suppose $n \geq 3$ and let $F_t : M \to \R^{n+1}$, $t \in (-\infty, 0]$, be an $n$-dimensional, complete, connected ancient solution to the mean curvature flow in $\R^{n+1}$ that is weakly convex, uniformly two-convex, and satisfies pointwise derivative estimates: $|\nabla A| \leq \gamma_1 |H|^2$ and $|\nabla^2 A| \leq \gamma_2 |H|^3$. Then the solution is noncollapsed. 
\end{theorem}

Note that by work of Haslhofer and Kleiner \cite{HK17A} mean-convex ancient solutions of the mean curvature flow that are noncollapsed are consequently weakly convex. If we combine the main result above with the uniqueness results of \cite{BC18} and \cite{ADS18}, we have the following corollary. 

\begin{corollary}
Suppose $n \geq 3$ and let $F_t : M \to \R^{n+1}$, $t \in (-\infty, 0]$, be an $n$-dimensional, complete, connected ancient solution to the mean curvature flow in $\R^{n+1}$ that is weakly convex, uniformly two-convex, and satisfies pointwise derivative estimates: $|\nabla A| \leq \gamma_1 |H|^2$ and $|\nabla^2 A| \leq \gamma_2 |H|^3$. Then the solution is either a family of shrinking round spheres, a family of shrinking round cylinders, a translating bowl soliton, an ancient oval, or a static flat hyperplane. 
\end{corollary}

Broadly, our work fits within an ongoing program aimed at characterizing self-similar and ancient solutions of the Ricci flow and the mean curvature flow. See for example \cite{Bre13, Has15, BL17, BC18, BC19, ADS18, ADS19}, to name some recent related works. Specifically, however, we use the theorem above and its corollary to classify singularity models of the mean curvature flow in higher codimension. 

The work of Brendle and Choi \cite{BC18} shows the only possible blow-up models at the first singular time for closed, \textit{embedded}, two-convex hypersurfaces evolving under the mean curvature flow are the spheres, cylinders, and bowls. The embeddedness assumption ensures that blow-up limits are noncollapsed (see \cite{Whi00, Whi03} or \cite{HK17A}), in addition to being ancient, weakly convex, and uniformly two-convex. Noncollapsing of blow-ups is an incredibly useful assumption, as demonstrated, for example, by the efficient works of Haslhofer and Kleiner \cite{HK17A,HK17B}. The work of Huisken and Sinestrari \cite{HS99,HS09}, however, shows in higher dimensions that noncollapsing is not necessary for the analysis of solutions when one has pointwise derivative estimates and a pinching estimate. For instance, see \cite{BL17}, where, under such assumptions Bourni and Langford proved a uniqueness theorem for translators of the mean curvature flow. 

By replacing the noncollapsed assumption in the main result of \cite{BC18} by an assumption of derivative estimates, we can classify blow-ups models for pinched solutions of the mean curvature flow in higher codimension, where embeddedness is no longer preserved. Our work also applies to immersed solutions in codimension one. Indeed, the blow-ups of closed, \textit{immersed}, two-convex solutions of the mean curvature flow are still ancient, weakly convex, and uniformly two-convex. In fact, by Theorem 5.3 in \cite{HS09}, blow-ups satisfy the cylindrical estimate $|A|^2 \leq \frac{1}{n-1}|H|^2$, which implies weak convexity and uniform two-convexity. By Theorems 6.1 and 6.3 in \cite{HS09}, blow-ups also satisfy $|\nabla A| \leq \gamma_1 |H|^2$ and $|\nabla^2 A| \leq \gamma_2 |H|^3$. Combining these results with the main theorem above allows us to drop the embeddedness assumption of Corollary 1.2 in \cite{BC18}. 

\begin{corollary}\label{codim1cor}
Let $n \geq 3$. Consider an arbitrary closed, immersed, two-convex hypersurface in $\R^{n+1}$, and evolve it by mean curvature flow. At the first singular time, the only possible blow-up limits are shrinking round spheres, shrinking round cylinders, and translating bowl solitons.
\end{corollary}

By replacing two-convexity with a stronger pinching assumption, we can show the above classification holds in higher codimension as well. Namely, we consider closed, $n$-dimensional initial data in $\R^{N}$ that satisfies a natural curvature pinching condition, $|H| > 0$ and $|A|^2 < c \, |H|^2$. This kind of pinching was first considered in $\cite{AB10}$, where the authors showed these inequalities are preserved by the flow if $c \leq \frac{4}{3n}$. Under these pinching conditions, Andrews and Baker \cite{AB10} and Nguyen \cite{Ngu18} have suitably extended to higher codimension many of the main ideas in the important and impactful works of Huisken \cite{Hui84} on the flow of convex hypersurfaces and Huisken and Sinestrari \cite{HS99,HS09} on the flow of two-convex hypersurfaces. For more discussion, see the introduction in \cite{Naf19}. 

At present, we are interested in the pinching condition with $c = \frac{1}{n-2}$. If $n \geq 8$, then $\frac{1}{n-2} \leq \frac{4}{3n}$. Suppose $F : M \times [0, T) \to \R^N$ is a smooth, closed, $n$-dimensional solution of the mean curvature flow in $\R^N$ initially satisfying $|H| > 0$ and $|A|^2 < \frac{1}{n-2} |H|^2$. In codimension one, this pinching condition implies two-convexity of the hypersurface. It was first studied by Nguyen in \cite{Ngu18}. In direct analogy with the work of Huisken and Sinestrari, Nguyen proved the following cylindrical and derivative estimates: 
\begin{enumerate}
\item[$\bullet$] (Cylindrical Estimate.) For every constant $\eta > 0$, there exists a constant $C_{\eta} < \infty$, depending only upon initial data, such that the estimate
\[
|A|^2 \leq \Big(\frac{1}{n-1} + \eta\Big)|H|^2 + C_{\eta}
\]
holds for all $t \in [0, T)$. 
\item[$\bullet$] (Derivative Estimates.) There exist constants $\gamma_1, \gamma_2, C_1, C_2 < \infty$, depending only upon the initial data, such that the estimates
\begin{align*}
|\nabla A| &\leq \gamma_1 |H|^2 + C_1, \\
 |\nabla^2 A| & \leq \gamma_2 |H|^3 + C_2
\end{align*}
hold for all $t \in [0, T)$. 
\end{enumerate}

In addition to the estimates above, in our earlier work \cite{Naf19} we proved a new planarity estimate for pinched solutions to the mean curvature flow in higher codimension. We considered a tensor, 
\[
\hat A_{ij} = A_{ij} - \frac{\langle A_{ij}, H \rangle}{|H|^2} H,
\]
which consists of the components of the second fundamental form that are orthogonal to the direction of the mean curvature vector. Under the pinching assumption, we showed that $\hat A$ vanishes if and only if the solution is codimension one. Then we proved the following estimate:
\begin{enumerate}
\item[$\bullet$] (Planarity Estimate.) There exists a constant $\sigma > 0$ and a constant $C < \infty$, depending only upon the initial data, such that the estimate
\[
|\hat A|^2 \leq C|H|^{2 - \sigma}
\]
holds for all $t \in [0, T)$. 
\end{enumerate}
We can include the dimensions $n = 5, 6$, and $7$ if we strengthen our pinching assumption. In fact, for $n \geq 5$, the cylindrical estimate holds if $c \leq \min\{ \frac{1}{n-2}, \frac{4}{3n}\}$; for $n \geq 2$ the derivative estimates hold if $c \leq \frac{4}{3n}$; and for $n \geq 5$ the planarity estimate holds if $c \leq \min\{\frac{3(n+1)}{2n(n+2)}, \frac{4}{3n}\}$. The constants $\frac{4}{3n}$ and $\frac{3(n+1)}{2n(n+2)}$ are technical constants that arise in the proofs of \cite{AB10} and \cite{Naf19}. If $n \leq 4$, then $\frac{4}{3n} \leq \frac{1}{n-1}$; so only spherical singularities can occur (see \cite{AB10}). 

In any case, if all three estimates above hold for a solution to the mean curvature flow, then the blow-up limits must be ancient, codimension one, and satisfy the estimates $|A|^2 \leq \frac{1}{n-1} |H|^2,  |\nabla A| \leq \gamma_1 |H|^2, |\nabla^2 A| \leq \gamma_2 |H|^3$, precisely as for immersed solutions in codimension one. By the main theorem above, this again gives the following classification. 

\begin{corollary}\label{codimNcor}
Let $n \geq 5$ and $N > n$. Let $c_n = \frac{1}{n-2}$ if $n \geq 8$ and $c_n = \frac{3(n+1)}{2n(n+2)}$ if $n = 5, 6,$ or $7$.  Consider a closed, $n$-dimensional solution to the mean curvature flow in $\R^N$ initially satisfying $|H| > 0$ and $|A|^2 < c_n|H|^2$. At the first singular time, the only possible blow-up limits are codimension one shrinking round spheres, shrinking round cylinders, and translating bowl solitons.
\end{corollary}

Let us briefly explain the arguments needed to show that Corollaries \ref{codim1cor} and \ref{codimNcor} follow from Theorem \ref{main}. For both immersed, two-convex solutions in codimension one and $\frac{1}{n-2}$-pinched solutions in higher codimension, blow-ups satisfy the cylindrical estimate $|A|^2 \leq \frac{1}{n-1} |H|^2$. Now by the strong maximum principle, a weakly convex ancient solution that is not strictly convex must split off a line. If a blow-up splits off a line, the cylindrical estimate implies the remaining principal curvatures are all equal and hence, by the Schur lemma, the blow-up must be a family of shrinking cylinders. In the immersed and codimension one setting, if the blow-up is compact, then it is convex, and the original flow must become convex. The result of Huiksen \cite{Hui84} then shows the blow-up is a family of shrinking spheres. In the $\frac{1}{n-2}$-pinched and higher codimension setting, if the blow-up is compact, it is $\frac{1}{n-1}$-pinched, $|A|^2 < \frac{1}{n-1} |H|^2$. Then the work of Nguyen \cite{Ngu18} shows the original flow is $\frac{1}{n-1}$-pinched, and the work of Andrews and Baker \cite{AB10} implies the blow-up is a family of shrinking spheres. Note that \cite{Hui84} and \cite{Ngu18} preclude the possibility of an ancient oval arising as a blow-up limit. 

The remaining case to consider is when the blow-up is noncompact and strictly convex, which is addressed by Theorem \ref{main} and the main result of \cite{BC18}. For the sake of generality, we have replaced cylindrical estimate in our assumptions with an assumption of uniform two-convexity. This way the convexity assumptions in our main theorem match the convexity assumptions in codimension one. To show uniform two-convexity suffices for the conclusions of the theorem, we prove, in Proposition \ref{optimal_two_convexity} below, that weakly convex, uniformly two-convex, ancient solutions satisfying pointwise derivative estimates are in fact $\frac{1}{n-1}$-two-convex. That is, if $\lambda_1$ and $\lambda_2$ denote the smallest two eigenvalues of the second fundamental form and $H$ denotes the mean curvature, then $\lambda_1 + \lambda_2 \geq \frac{1}{n-1}H$. 

Finally, let us discuss the proof of the main theorem. Broadly, we will prove it in two steps. In the first step, we will adapt the tools and ideas developed by Huisken and Sinestrari to our ancient and convex setting. We will show that when the ancient solution is strictly convex (and not a family of shrinking round spheres), it has the structure of a long tube with either one or two convex caps attached, depending upon whether the solution is noncompact or compact. Many ideas carry over without much change; some are even a bit simpler in our setting (also here we do not deal with the difficulties arising from surgery). The structure theorems give diameter, mean curvature, and uniform convexity estiamtes independent of time. Once this is known, we can show each time slice of the ancient solution is $\alpha$-noncollapsed for a uniform choice of $\alpha$. 

The organization of this paper is as follows. In Section 2, we collect the various notations, definitions, and auxiliary results that will be used in subsequent sections. In Section 3, we prove that when the ancient solution is noncompact, it has the structure of a long tube with a convex cap attached. In Section 4, we prove a similar structure theorem for when the solution is compact. The work in these three sections follows the pioneering work of Huisken and Sinestrari in \cite{HS09}. In Section 5, we show the ancient solution is noncollapsed. In the Appendix, we include additional details for the proof of Proposition \ref{optimal_two_convexity}. \\

\textbf{Acknowledgments.} I am very grateful to my advisor, Simon Brendle, for his guidance, his encouragement, and for suggesting this problem. 


\section{Preliminaries}
In this section, we give some definitions and auxiliary results for the proof of our main theorem. Many of the statements and proofs in this section are to a certain extent standard after \cite{HS09} and only require reasonable adaptations of the analogous statements and proofs. For the convenience of the reader, we include these adaptations and some additional details below. 

Let us begin with some notation. Let $F : M \times (-\infty, 0] \to \R^{n+1}$ denote a (possibly immersed) ancient solution of the mean curvature flow. Let $(p, t) \in M \times (-\infty, 0]$ be a spacetime point. Since we are working with a hypersurface, from now on we will let $H$ denote the scalar mean curvature, as opposed to the mean curvature vector; $\nu$ denote the outward pointing normal vector; and $h = \langle A, -\nu \rangle$ denote the second fundamental form. Let $\lambda_1 \leq \dots \leq \lambda_n$ denote the principal curvatures. We say the solution is weakly convex if $\lambda_1 \geq 0$; strictly convex if $\lambda_1 > 0$; uniformly convex if $\lambda_1 \geq \beta H$ for some $\beta > 0$; and analogously for the notion of two-convexity. Let $g(t)$ denote the induced metric on $M$ by the immersion $F(\cdot, t)$ and $B_{g(t)}(p, r) \subset M$ the intrinsic ball of radius $r$ centered at $p$. We will also be interested in parabolic neighborhoods. Following the notation introduced on pp.189-190 in \cite{HS09}, we define $P(p, t, r, \theta)$ to be the set of space time points $(q, s)$ such that $q \in B_{g(t)}(p, r)$ and $s \in [t - \theta, t]$. For the purposes of rescaling, we also define 
\[
\hat P(p, t, L, \theta) := P\Big(p, t, \frac{n-1}{H(p,t)}L, \Big(\frac{n-1}{H(p,t)}\Big)^2\theta\Big).
\]

Throughout this section and following sections, we will only consider solutions of the mean curvature flow satisfying the hypotheses of Theorem \ref{main}.
\begin{definition} An $n$-dimensional ($n \geq 3)$ ancient solution $F : M  \times (-\infty, 0] \to \R^{n+1}$ satisfies $(\ast)$ if:
\begin{enumerate}
\item[$\bullet$] The solution is connected, complete, strictly convex, embedded, and uniformly two-convex.
\item[$\bullet$] The solution satisfies the pointwise derivative estimates $|\nabla A| \leq \gamma_1 |H|^2$ and $|\nabla^2 A| \leq \gamma_2 |H|^3$.
\end{enumerate}
\end{definition}
Notice we assume that the ancient solution is strictly convex. If the convexity is not strict, then by the strong maximum principle, the solution must split a line. It follows that the cross-section is a uniformly convex, complete, ancient solution of the mean curvature flow. Now if the cross-section contains a point where $\lambda_1 = 0$, then the entire solution must be flat and hence a static hyperplane. If the cross-section is strictly uniformly convex, then by work of Hamilton \cite{Ham94} and Huisken-Sinestari \cite{HS15}, the cross-section is a family of shrinking round spheres, and hence the ancient solution is a family of shrinking round cylinders. Of course both the static plane and shrinking round cylinder are noncollapsed, so it suffices to prove the main theorem assuming strict convexity. 

We also make an extra assumption of embeddedness in the definition above. In higher dimensions, strict convexity and completeness imply $F_t(M)$ is the boundary of a convex body in $\R^{n+1}$ (see \cite{Sac60}). The boundary of a strictly convex body in $\R^{n+1}$ is homeomorphic to either $S^n$ or $\R^n$, and in particular is simply connected. Since $F_t$ is a covering, this implies it is an embedding, and so we may assume this in $(\ast)$ as well. In this case, it is equivalent to work with the level sets $M_t = F(M, t)$, for $t \in (-\infty , 0]$. By abuse of notation, we will sometimes identify the points $p \in M$ and $F(p, t) \in M_t$. Let $\mathcal M = \{M_t\}_{t \in (-\infty, 0]}$. For brevity, we will say $\mathcal M$ (or sometimes $F : M \times (-\infty, 0] \to \R^{n+1}$) with these properties satisfies $(\ast)$. 

We will use the following definition to characterize neck regions. Since we are not doing surgery, we will not need precise parametrizations of neck regions, as originally introduced by Hamilton in \cite{Ham97} and used extensively in \cite{HS09}. (One great property of Hamilton's (intrinsic) constant mean curvature foliation of neck regions is that it is canonical! However, in the extrinsic setting, the following definition is perhaps a bit simpler). 
\begin{definition}\label{necks}
Suppose $F: M \to  \R^{n+1}$ is an embedding of a convex hypersurface. Given constants $\varepsilon, L > 0$, we say a point $p \in M$ lies at the center of an $(\varepsilon, L)$-neck if, after rescaling so that $(n-1)H(p)^{-1} = 1$, there exist an embedded round cylinder $\Sigma := S^{n-1} \times [-L, L] \subset \R^{n+1}$ (of radius 1) and a function $u : \Sigma \to \R$, with the following properties:
\begin{enumerate}
\item[$\bullet$] We have $\{x + u(x) \nu_{\Sigma}(x) : x \in \Sigma \} \subset F(M)$ and $||u||_{C^{10}(\Sigma)} \leq \varepsilon$. 
\item[$\bullet$] $F(p) \in \{x + u(x) \nu_{\Sigma}(x) : x \in S^{n-1} \times \{0\}\, \}$, i.e. the point $p$ lies on the central sphere. 
\end{enumerate}
Let $N := F^{-1}(\{x + u(x) \nu_{\Sigma}(x) : x \in \Sigma \})$. We will say $N$ is an $(\varepsilon, L)$-neck and has length $2L$. Up to a choice of orientation, there exists a unit vector $\omega \in S^n$ which is tangent to the axis of $\Sigma$. This defines a height function $y : M \to \R$ by $y(q) = \langle F(q) - F(p), \omega \rangle$. The sets $S_y := \{q \in N : y(q) = y \}$ for $y \in [-L, L]$ are the cross-sectional spheres of $N$. The axis of the neck $N$ refers to either of the two unit vectors parallel to the axis of $\Sigma$. 
\end{definition}

In the terminology introduced in Section 3 of \cite{HS09}, Definition \ref{necks} defines a ``geometric neck''; i.e. a neck that is parametrized by a cylinder (in this case, as a graph). This ensures the restriction of the embedding $F$ to the region $N \subset M$ is close to the standard embedding of a cylinder into $\R^{n+1}$. For the detection of necks however, it is much simpler to check if the curvature is close to that of a cylinder. These ideas go back to Hamilton's work on necks in \cite{Ham97}. We will use the following proposition, which is a direct consequence of Proposition 3.5 in \cite{HS09}. 

\begin{proposition}\label{curvature_necks}
Suppose $F : M \to \R^{n+1}$ is an embedding of a convex hypersurface. Given constants $\varepsilon_0 \in (0, \frac{1}{n})$ and $L \geq 100$, there exists $\varepsilon_1 \in (0, \varepsilon_0)$, depending only upon $n$, $L$, and $\varepsilon_0$, such that the following holds. Suppose $p \in M$ satisfies: 
\begin{enumerate}
\item[$\bullet$] $\lambda_1(p) \leq \varepsilon_1 H(p)$ and $\lambda_n(p) - \lambda_2(p) \leq \varepsilon_1 H(p)$;
\item[$\bullet$] For every $q\in B_g(p, (L+10)\frac{n-1}{H(p)})$, $\sum_{k =1}^8 H(p)^{-k-1} |\nabla^k h(q) | \leq \varepsilon_1$. 
\end{enumerate}
Then $p$ lies at the center of an $(\varepsilon_0, L)$-neck $N$.
\end{proposition}

In the language of \cite{HS09}, a point $p \in M$ that satisfies the two properties of the above proposition is said to lie at the center of an $(\varepsilon_1, L + 10)$-extrinsic ``curvature neck''. The above proposition shows there is little difference between curvature necks and the geometric necks defined above. 

Next, we give two standard lemmas concerning the control of curvature. The first shows that the curvature at one point controls the curvature of all points in a suitable parabolic neighborhood (cf. Lemma 6.6 in \cite{HS09}).

\begin{lemma}\label{curvature_control}
Suppose $\mathcal M$ satisfies $(\ast)$. Let $(p_0, t_0)$ be a spacetime point. Then there exists a constant $\hat r := \hat r(n, \gamma_1, \gamma_2) > 0$ such that for every $(p,t) \in \hat P(p_0, t_0, \hat r, \hat r^2)$ we have
\[
\frac{1}{4} H(p_0, t_0) \leq H(p,t) \leq 4 H(p_0, t_0).
\] 
\end{lemma}
\begin{proof}
By a rescaling, we may assume $H(p_0, t_0) = n-1$ so that $\hat P(p_0, t_0,  r,  \theta) = P(p_0, t_0,  r, \theta)$.  Recall that the mean curvature satisfies the evolution equation $\pdv{t} H = \Delta H + |h|^2 H$. From the estimates $|\nabla h| \leq \gamma_1 H^2$, $|\nabla^2 h| \leq \gamma_2 H^3$, and $|h| \leq H$, we get the estimates $|\nabla H| \leq c_1 H^2$ and $|\pdv{t} H| \leq c_2H^3$ for some constants $c_1:= c_1(n, \gamma_1)$ and $c_2 := c_2(n, \gamma_2)$. Consider a point $p \in M$ and let $\alpha : [0, 1] \to M$ be a minimizing geodesic with respect to $g(t_0)$ beginning at $p_0$ and ending at $p$. Assume $|\alpha'(s)| =  d_{g(t_0)}(p, p_0) \leq r_1$ for some $r_1 > 0$. By our gradient estimate,
\[
\Big|\frac{d}{ds} H^{-1}(\alpha(s), t_0)\Big| = \frac{\big|\langle \nabla H (\alpha(s), t_0), \alpha'(s)\rangle\big|}{H(\alpha(s), t_0)^2} \leq c_1.
\]
Integrating along the geodesic gives
\[
\Big|\frac{1}{H(p, t_0)} - \frac{1}{H(p_0, t_0)}\Big| \leq c_1 r_1. 
\]
With $H(p_0, t_0) = n-1$, this implies 
\[
\frac{1}{1+ c_1(n-1)r_1} \leq \frac{H(p, t_0)}{H(p_0, t_0)} \leq  \frac{1}{1 - c_1 (n-1)r_1}. 
\]
We choose $r_1$ sufficiently small depending upon $n$ and $c_1$ so  $\frac{1}{2}H(p_0, t_0) \leq H(p,t_0) \leq 2 H(p_0, t_0)$ for all $p \in B_{g(t_0)}(p_0, r_1)$. Next, fix some point $p \in B_{g(t_0)}(p_0, r_1)$. Consider a time $t \in [t_0 - r_2, t_0]$ for some $r_2 > 0$.  Let $\beta(s) = (1-s)t_0 + st$ for $s \in [0,1]$. By our estimate for the time derivative of $H$,
\[
\Big| \frac{d}{ds} H^{-2}(p, \beta(s))\Big| = \frac{\big|(\pdv{t} H)(p, \beta(s))\big|}{H(p, \beta(s))^3}|t - t_0| \leq c_2.
\]
Integrating gives
\[
\Big|\frac{1}{H(p, t)^2} - \frac{1}{H(p, t_0)^2}\Big| \leq c_2r_2. 
\]
Using $H(p, t_0) \in[\frac{1}{2}(n-1), 2(n-1)]$, we can rewrite this expression as
\[
\frac{1}{1 + 4c_2(n-1)^2 r_2} \leq \frac{H(p, t)^2}{H(p, t_0)^2} \leq \frac{1}{1 - 4c_2(n-1)^2r_2}
\]
As before, we can choose $r_2$ sufficiently small so that $\frac{1}{2} H(p, t_0) \leq H(p, t) \leq 2 H(p, t_0)$ for all $t \in [t_0 - r_2, t_0]$. Combining this with the estimate on $B_{g(t_0)}(p_0, r_1)$ and choosing $\hat r$ small depending upon $r_1$ and $r_2$ gives the desired estimate on $P(p_0, t_0, \hat r, \hat r^2)$. 
\end{proof}

Interior estimates now give pointwise estimates for all higher order derivatives of the second fundamental form (cf. Corollary 6.4 in \cite{HS09}). 

\begin{lemma}\label{derivative_estimates}
Suppose $\mathcal M$ satisfies $(\ast)$. For all nonnegative integers $k, l$ there exist constants $\gamma_{k,l} := \gamma_{k,l}(n, \gamma_1, \gamma_2) < \infty$ such that the pointwise estimates,
\[
\Big|\pdv{t}^l \nabla^k h(p_0, t_0)\Big| \leq \gamma_{k,l}H^{2l + k + 1}(p_0, t_0),
\]
hold for every $(p_0, t_0) \in M \times (-\infty, 0]$. 
\end{lemma}
\begin{proof}
By standard interior estimates for the mean curvature flow (e.g. \cite{EH91}), if $|h(p, t)| \leq r^{-1}H(p_0, t_0)$ in the parabolic neighborhood $\hat P(p_0, t_0, r, r^2)$, then there exists a constant $C_k$ depending only upon the dimension $n$ and $k$, such that
\[
|\nabla^k h(p_0, t_0)| \leq C_k r^{-k-1} H(p_0, t_0)^{k+1}
\]
Fix a point $(p_0, t_0)$. By the previous lemma, we can find $r \in (0, \frac{1}{4})$, depending only $\gamma_1, \gamma_2$ and $n$, such that 
\[
|h(p,t)| \leq H(p,t) \leq 4 H(p_0, t_0) \leq r^{-1} H(p_0, t_0)
\]
for every $(p, t) \in \hat P(p_0, t_0, r, r^2)$, which implies
\[
|\nabla^k h(p_0, t_0)| \leq C_k r^{-k -1}H(p_0, t_0)^{k+1}. 
\]
This proves the pointwise estimates when $l = 0$. For $l > 0$, we recall that the evolution equation for $\nabla^m h$ is of the form
\[
\pdv{t} \nabla^m h = \Delta \nabla^m h  + \sum_{i + j + k = m } \nabla^i h \ast \nabla^j h \ast \nabla ^k h.
\]
The estimates for time derivatives of $\nabla^k h$ follow by induction, the evolution equation above, and a straightforward commutation identity for $\pdv{t}$ and $\Delta$.
\end{proof}

Each of our model solutions (the spheres, cylinders, bowls, and ovals) is $\frac{1}{n-1}$-two-convex. With the higher derivative estimates established, we can now show that $\mathcal M$ is $\frac{1}{n-1}$-two-convex. Several steps in our proof are inspired by the arguments given in \cite{BH19}. 

\begin{proposition}\label{optimal_two_convexity}
Suppose $\mathcal M$ satisfies $(\ast)$. Then $\mathcal M$ is $\frac{1}{n-1}$-two-convex. 
\end{proposition}
\begin{proof}
Recall that $0 < \lambda_1 \leq \lambda_2 \leq \cdots \leq \lambda_n$ denote the eigenvalues of the second fundamental form, $h$, with multiplicity. 
Let 
\[
\beta = \inf_{(p, t) \in M \times (-\infty, 0]} \frac{\lambda_1(p, t) + \lambda_2(p, t)}{H(p,t)}.
\]
We have assumed $\mathcal M$ is uniformly two-convex, so $\beta > 0$. Our goal is to show $\beta \geq \frac{1}{n-1}$. To that end, consider a sequence of points $(p_j, t_j) \in M \times (-\infty, 0]$ such that $(\lambda_1(p_j, t_j) + \lambda_2(p_j, t_j))/H(p_j, t_j) \to \beta$ as $j \to \infty$. Let $Q_j = (n-1)H(p_j, t_j)^{-1}$ and let $\hat r$ be the constant appearing in Lemma \ref{curvature_control}. Suppose our solution $\mathcal M$ is given by the family of embeddings $F : M \times (-\infty, 0] \to \R^{n+1}$. For each $j$, we first consider the restriction of $F$ to the parabolic neighborhood $\hat P(p_j, t_j, \hat r, \hat r^2)$. Next we perform a parabolic rescaling and spacetime translation to define solutions $F_j$ by 
\[
F_j(p, \tau) = \frac{1}{Q_j} \big[ F\big(p, Q_j^2 \tau + t_j) - F(p_j, t_j) \big]. 
\]
Let ${\mathcal M}^{(j)}$ denote the resulting flow. Let $g_j, h_j,$ and $H_j$ denote the metric, second fundamental form, and mean curvature of $\mathcal M^{(j)}$. The solution $F_j$ is defined in the parabolic neighborhood $P_j(\hat r) := P(p_j, 0, \hat r, \hat r^2) = B_{g_j(0)}(p_j, \hat r) \times [-\hat r^2, 0]$. By construction, $H_j(p_j, 0) = n-1$ and $F_j(p_j, 0) = 0 \in \R^{n+1}$. Lemmas \ref{curvature_control} and \ref{derivative_estimates} imply that
\[
\sup_{P_j(\hat r)} \,\Big|\pdv{t}^l \nabla^k h_j \Big|\leq C(k, l, n, \gamma_1, \gamma_2). 
\] 
This gives local uniform control in any $C^k$ norm over solutions $F_j$ in the parabolic neighborhoods $P_j(\hat r)$. Standard compactness results (e.g. \cite{Coo11}) then imply the solutions $F_j$ converge smoothly and uniformly on compact subsets of $P_j(\hat r)$ to a locally defined limit flow $\hat{\mathcal M}$ defined in a parabolic neighborhood $\Omega \times [-T, 0]$ for some $T > 0$. We may assume $\Omega$ is a smooth domain in $\R^n$ and that the points $p_j$ converge to $p \in \Omega$. Let us denote the limiting geometric quantities by $g$, $h$, and $H$. 

We now analyze the limit flow in $\Omega \times [-T, 0]$. Let $u(q,t) := (\lambda_1 + \lambda_2 - \beta H)(q, t)$. The limit is weakly convex. Moreover, $u \geq 0$ and our choice of points $(p_j, t_j)$ implies that $u(p, 0) = 0$. By the strong maximum principle, these conditions imply $u$ vanishes identically in $\Omega \times (-T, 0)$. To see this, consider the $(0, 2)$-tensor $U_{ik} := h_{ik} - \frac{\beta}{2}Hg_{ik}$. The sum of the two smallest eigenvalues of $U$ is $u$. We argue by contradiction: suppose for some $\tau \in (-T, 0)$, there exists a point $q_0$ such that $u(q_0, \tau) > 0$. Recalling the evolution equations for $H$ and $h^i_k$, we have 
\[
\pdv{t}U^i_k = \Delta U^i_k + |A|^2U^i_k.
\]
Since $u(q_0, \tau) > 0$ and $u(q, \tau) \geq 0$ for all $q \in\Omega$, we can find a smooth nonnegative function $f_0$ defined on $\Omega$ such that $f_0(q) \leq u(q, \tau)$ for all $q \in \Omega$, $f_0(q_0) \geq \frac{1}{2}u(q_0, \tau)$, and $f(q) = 0$ for all $q \in \d \Omega$. Now let $f$ be a solution to the heat equation $\pdv{t}f = \Delta f$ with initial condition $f(q, \tau) = f_0(q)$ and boundary condition $f(q, t) = 0$ for all $q \in \d \Omega$. Since $f_0(q_0) > 0$, the strict maximum principle for scalar equations implies $f > 0$ in $\Omega \times (\tau, 0]$. Moreover, the tensor $\tilde U_{ik} := U_{ik} - \frac{f}{2} g_{ik}$ satisfies the equation 
\begin{align*}
\pdv{t}\tilde U^i_k &= \Delta \tilde U^i_k + |A|^2 \tilde U^i_k + |A|^2 f \delta^i_k \\
& \geq \Delta \tilde U^i_k + |A|^2 \tilde U^i_k.  
\end{align*}
The sum of the first two eigenvalues of $\tilde U$ is $u - f$. At the initial time $\tau$, $\tilde U$ is weakly two-convex since by construction $u(\cdot, \tau) - f_0 \geq 0$. By the weak maximum principle for tensors, weak two-convexity of $\tilde U$ must hold in $\Omega \times [\tau,0]$. However, in this case we conclude $u(p, 0) \geq f(p, 0) > 0$, a contradiction. 

We have shown that $\lambda_1 + \lambda_2 \equiv \beta H$ in $\Omega \times (-T, 0)$. Fix a time $\tau \in (-T, 0)$ and let us only consider $g, h$ and $H$ at the time $\tau$. For each $q \in \Omega$, let us consider the set of orthonormal two-frames 
\[
E_{q}:= \{ \{e_1, e_2\} \subset T_q\Omega : |e_1| = |e_2| = 1, \; \langle e_1, e_2 \rangle = 0, \; h(e_1, e_1) + h(e_2, e_2) = \lambda_1(q) + \lambda_2(q)\}.
\]
The strict maximum principle for tensors implies that the set $E_q$ is invariant under parallel transport (with respect to $g(\tau)$). The proof follows directly from the results of Chapter 9 of \cite{Bre10}. For the convenience of the reader, we include the details of this argument in the Appendix. We next construct a parallel subbundle of $T\Omega$ out of the eigenspaces of $h$. To that end, for each $q \in \Omega$, consider the eigenspaces $V_{1,q} := \mathrm{ker}(h - \lambda_1 g)$ and $V_{2,q} := \mathrm{ker}(h - \lambda_2 g)$. Define a vector space $\tilde E_q := V_{1,q} + V_{2,q}$. Note it is possible $\lambda_1(q) = \lambda_2(q)$, in which case $\tilde E_q = V_{1,q}$. \\

\noindent \textbf{Claim:} $v \in \tilde E_q$ if and only if there exists $\{e_1, e_2\} \in E_q$ such that $v \in \mathrm{span}\{e_1, e_2\}$. 

\begin{proof}[Proof of claim.] For simplicity, since the claim does not depend upon $q$, let us suppress it in our notation. First we consider the case $\lambda_1 = \lambda_2$. In this case $\dim V_1 \geq 2$. The identity defining $E$ becomes $h(e_1, e_1) + h(e_2, e_2) = 2\lambda_1$, which implies $h(e_1, e_1) = h(e_2, e_2) = \lambda_1$. This shows $\mathrm{span}\{e_1, e_2\} \subset \tilde E$. Since $\dim V_1 \geq 2$, the converse is also clear. 

 Now assume $\lambda_2 > \lambda_1$. In this case, $h$ has a unique $\lambda_1$-eigenvector, which we denote by $e_1$. First suppose $v \in \tilde E$. Let $v_1 = \langle v, e_1 \rangle$ and write $v = \tilde v + v_1 e_1$, where $\tilde v \in V_2$. If $\tilde v =0$, then we can take $e_2$ to be any $\lambda_2$-eigenvector and we will have $v \in \mathrm{span}\{e_1, e_2\}$. If not, then take $e_2 = |\tilde v|^{-1} \tilde v$. Again, this gives $v \in \mathrm{span}\{e_1, e_2\}$ and in either case $\{e_1, e_2\} \in E$. 

Now for the other direction, suppose $\{v, w\} \in E$. Then $|v|^2 = |w|^2 = 1$, $\langle v, w \rangle = 0$, and $h(v,v) + h(w,w) = \lambda_1+ \lambda_2$. As above, let $v_1 = \langle v, e_1\rangle$ and $w_1 = \langle w, e_1\rangle$ and write $v = \tilde v + v_1 e_1$ and $w = \tilde w + w_1 e_1$. We must show $\tilde v, \tilde w \in V_2$. The case when either $\tilde v$ or $\tilde w$ is zero is straightforward, so we may assume they are not. Our assumptions on $v$ and $w$ imply $1 = |\tilde v|^2 + v_1^2 = |\tilde w|^2 + w_1^2$ and $v_1w_1 + \langle \tilde v, \tilde w\rangle = 0$. We compute 
\begin{align*}
h(\tilde v, \tilde v)  - \lambda_2|\tilde v|^2 + h(\tilde w, \tilde w) -\lambda_2 |\tilde w|^2 &= h(v,v) + h(w,w) - \lambda_1(v_1^2 + w_1^2)- \lambda_2(|\tilde v|^2 + |\tilde w|^2) \\
& = \lambda_1 + \lambda_2 -\lambda_1(v_1^2 + w_1^2)-  \lambda_2(|\tilde v|^2 + |\tilde w|^2)\\
& = \lambda_1(v_1^2 + |\tilde v|^2) + \lambda_2(w_1^2 +|\tilde w|^2) -\lambda_1(v_1^2 + w_1^2)-  \lambda_2(|\tilde v|^2 + |\tilde w|^2) \\
& = (\lambda_2 - \lambda_1)(w_1^2  - |\tilde v|^2). 
\end{align*}
Now $\lambda_2 = \inf\{ |x|^{-2} h(x,x) : \langle x, e_1 \rangle = 0, \, x \neq 0 \}$. Therefore $h(\tilde v, \tilde v)  \geq \lambda_2|\tilde v|^2$ and $h(\tilde w, \tilde w)  \geq \lambda_2|\tilde w|^2$. If either of these inequalities is strict, then $(\lambda_2 - \lambda_1)(w_1^2 - |\tilde v|^2) > 0$ and hence $w_1^2 > |\tilde v|^2$. Similarly $v_1^2 > |\tilde w|^2$. However, we also have $v_1w_1 + \langle \tilde v, \tilde w\rangle = 0$, which implies $v_1^2w_1^2 = \langle \tilde v, \tilde w\rangle^2 \leq |\tilde v|^2 |\tilde w|^2$. Thus, neither inequality can be strict and so $|\tilde v|^{-2}h(\tilde v, \tilde v) = |\tilde w|^{-2} h(\tilde w, \tilde w) = \lambda_2$. We conclude that both $\tilde v, \tilde w$ are in $V_2$ and hence $\mathrm{span}\{v, w \}\subset \tilde E$. 
\end{proof}

Returning to the proof of the proposition, an immediate consequence of the claim is that $\tilde E_q$ is invariant under parallel transport. In particular, the dimension of the vector spaces $\tilde E_q$ is constant and so $\tilde E := \bigcup_{q\in \Omega} \tilde E_q$ is a parallel vector subbundle of $T\Omega$. The classical theorem of de Rham implies $(\Omega, g(\tau))$ splits locally as an isometric product $\Omega_1^k \times \Omega_2^{n-k}$ of smaller dimensional spaces, where $k$ denotes the rank of the bundle $\tilde E$. Since the embedding of $(\Omega, g(\tau))$ is strictly two-convex, it has strict positive isotropic curvature as an intrinsic manifold. This implies (see \cite{Bre10}) the only possible splittings are $k = n-1$ or $k = n$ (no splitting). 

If $k = n -1$, then the eigenvector $e_n$, corresponding to the largest eigenvalue of $h$, is a parallel vector. Because the vector is parallel, $\Ric(e_n, e_n) = 0$. On the other hand, the Gauss equation implies that $\Ric(e_n, e_n) = \lambda_n(H - \lambda_n)$, which is not zero. Thus $k = n-1$ cannot occur. If $k = n$, then $\lambda_n = \lambda_2$. Since $\beta \leq 1$, we have
\[
\lambda_1 + \lambda_2 = \beta H = \beta \lambda_1 + \beta (n-1) \lambda_2  \leq \lambda_1+ \beta (n-1) \lambda_2.
\]
Therefore, $(1 - \beta (n-1) )\lambda_2 \leq 0$. Since $\lambda_2 > 0$, we conclude $\beta \geq \frac{1}{n-1}$, as was to be shown. 

\end{proof}

The final four results of this section concern the detection of necks. The first is a rephrasing of Theorem 7.14 in \cite{HS09} with a variation on its proof. 

\begin{theorem}[Theorem 7.14 in \cite{HS09}] 
Assume $n \geq 2$. Given constants $c_1, \eta_0 > 0$, we can find constants $\hat a$ and $\hat b$ with the following property. Let $F : M \to \R^{n+1}$ be a complete, connected, immersed hypersurface in $\R^{n+1}$ with $H > 0$. Suppose that $p \in M$. Moreover, suppose that $|\nabla H| \leq c_1H^2$ and $\lambda_1 \geq \eta_0 H$ for each point in the set $U := \{ x \in M : d_g(p, x) < \hat a H(p)^{-1}, H(x) > \hat b^{-1} H(p)\}$. Then $U = M$; in particular, M is compact.
\end{theorem}
\begin{proof} Choose $\hat a$ so that $\frac{\eta_0}{c_1} \mathrm{log}(1 + \frac{c_1\hat a}{100} ) > 2\pi$. Let $\hat b := 2+ 100c_1\hat a$. Then $U$ is nonempty. 

Let $f(x) := \langle F(x) - F(p), \omega \rangle $ where $\omega := - \nu(p)$. Clearly, $f$ has a strict local minimum at $p$. For each $s > 0$, we denote by $U_s$ the connected component of $\{f < s\}$ that contains the point $p$. Let $s_{\ast}$ denote the supremum of all $s$ with the property that $U_s$ contains no critical points of $f$ other than $p$.

Clearly, $U_s \subset U$ if $s > 0$ is sufficiently small. Let $s_0$ denote the supremum of all $s \in (0, s_{\ast}]$ with the property that $U_s \subset U$. We claim that $U_s \subset \{x \in M : d_g(p, x) \leq \frac{1}{2} \hat a H(p)^{-1}, H(x) \geq 2 \hat b^{-1} H(p)\}$ for $s \in (0, s_0)$. In other words, the sets $U_s$ are contained in $U$ until the next height at which $f$ has a critical point. To see this, fix a real number $s \in (0, s_0)$ and an arbitrary point $x \in U_s$. Let $\gamma(r)$ denote the integral curve of the vector field $- \frac{\nabla f}{|\nabla f|} = - \frac{\omega^\top}{|\omega^\top|}$ starting at $x$. Clearly, $\gamma$ converges to $p$ since $p$ is the only critical point of $f$ in $U_s$. Note that $\gamma$ is parametrized by arc length. We assume that $\gamma(r)$ is defined for $r \in [0, r_0)$ and satisfies $\gamma(0) = x$ and $\gamma(r) \to p$ as $r \to r_0$. Since the path $\gamma$ is contained in $U$, we know that $|\nabla H| \leq c_1 H^2$ and $\lambda_1 \geq \eta_0 H$ at each point on $\gamma$. Integrating the gradient estimate gives
\[
H(\gamma(r)) \geq \frac{H(p)}{ 1 + c_1 (r_0 - r)H(p)}.
\]
Uniform convexity implies
\[
\frac{d}{dr} \langle \nu, \omega \rangle = - \frac{h(\omega^\top, \omega^\top)}{|\omega^\top|} \leq - \eta_0 H |\omega^\top| = - \eta_0 H \sqrt{1 - \langle \omega, \nu \rangle^2}.
\]
This gives
\[
\frac{d}{dr} \mathrm{arcsin}(\langle \omega, \nu \rangle ) \leq - \eta_0 H \leq - \eta_0 \frac{H(p)}{ 1 + c_1 (r_0 - r)H(p)}.
\]
Since $\mathrm{arcsin}(\langle \omega, \nu \rangle )$ takes values in the interval $[-\frac{\pi}{2}, \frac{\pi}{2}]$, it follows that 
\[
\int_0^{r_0} \eta_0 \frac{H(p)}{ 1 + c_1 (r_0 - r)H(p)} \, dr \leq 2\pi.
\]
This gives $\frac{\eta_0}{c_1} \mathrm{log}(1 + c_1r_0H(p)) \leq 2\pi$. In view of our choice of $\hat a$, we obtain $r_0 H(p) \leq \frac{1}{2}\hat a$. This implies $d_g(p, x) \leq r_0 \leq \frac{1}{2} \hat a H(p)^{-1}$ and, in view of our choice of $\hat b$, $H(x) \geq H(p)( 1 + c_1 r_0H(p))^{-1}\geq 2\hat b^{-1} H(p)$. This shows that $U_s \subset  \{x \in M : d_g(p, x) \leq \frac{1}{2} \hat a H(p)^{-1}, H(x) \geq 2 \hat b^{-1} H(p)\} \subset U$ for all $s \in (0, s_0)$. Consequently, $s_0 = s_\ast$, as originally claimed. 

For each $s \in (0 , s_\ast)$, $U_s$ is diffeomorphic to $B^n$. Moreover, in view of the maximality of $s_\ast$, there exists a sequence $s_j \nearrow s_\ast$ and a sequence of points $q_j \in \d U_{s_j}$ such that $|\nabla f(q_j)| \to 0$. Therefore, $q_j \to q_\ast$ where $q_\ast$ is a critical point of $f$. Clearly, $q_\ast \in U$ and $f(q_\ast) = s_\ast$. Since $q_\ast \in U$, the second fundamental form at $q_\ast$ is positive definite. This implies that the Hessian of $f$ at $q_\ast$ is either positive or negative definite. Therefore, $f$ has a strict local maximum or minimum at $q_\ast$. Since $q_j \to q_\ast$ and $f(q_j) \nearrow f(q_\ast)$, $f$ cannot have a strict local minimum at $q_\ast$. Consequently, $f$ must have a strict local maximum at $q_\ast$. For each $s \in (0, s_\ast)$ we denote by $V_s$ the connected component of $\{f > s\}$ which contains the point $q_\ast$. Since $q_j \to q_\ast$, we conclude that $q_j \in \d V_{s_j}$ for $j$ large. If $j$ is sufficiently large, then $V_{s_j}$ is diffeomorphic to $B^n$. 

To summarize, both $\d U_{s_j}$ and $\d V_{s_j}$ are connected components of $\{f = s_j\}$ (here we use that $n \geq 2)$ and both $\d U_{s_j}$ and $\d V_{s_j}$ contain the point $q_j$. Therefore $\d U_{s_j} = \d V_{s_j}$. Since $M$ is connected, we conclude $\bar U_{s_j} \cup \bar V_{s_j} = M$. Since $\mathrm{diam}_g(V_{s_j}) \to 0$, it follows that $\bigcup_j U_{s_j}$ is dense in $M$. Since $\bigcup_j U_{s_j}$ is contained in the set $\{x \in M : d_g(p, x) \leq \frac{1}{2} \hat a H(p)^{-1}, H(x) \geq 2 \hat b^{-1} H(p)\}$, we conclude $M$ itself is contained in the set $\{x \in M : d_g(p, x) \leq \frac{1}{2} \hat a H(p)^{-1}, H(x) \geq 2 \hat b^{-1} H(p)\}$. This completes the proof. 
\end{proof}

In our setting, we are concerned with a complete, connected, convex hypersurface that satisfies the pointwise gradient estimate $|\nabla H| \leq c_1 H^2$ everywhere (with constant $c_1 = n \gamma_1$). We will use the theorem above to find necks at a controlled distance from points which do lie on necks. Specifically, we will use the following immediate corollary. 

\begin{corollary}\label{nearby_neck}
Given constants $\eta, \gamma_1 > 0$, there exist positive constants $\hat a, \hat b$, depending only upon $n$, $\eta$, and $\gamma_1$, with the following property. Suppose $F : M \to \R^{n+1}$ is an embedding of a complete, connected, convex hypersurface satisfying the pointwise derivative estimate $|\nabla h| \leq \gamma_1 H^2$. Suppose $p_0 \in M$ is a point such that $\lambda_1(p_0) > \eta H(p_0)$. Then either, $M = B_g(p_0, \hat a H(p_0)^{-1})$ or there exists a point $p_1 \in B_g(p_0, \hat a H(p_0)^{-1})$ where $\lambda_1(p_1) \leq \eta H(p_1)$. Moreover, in either case every point $p \in B_g(p_0, \hat a H(p_0)^{-1})$ satisfies $H(p) \geq \hat b^{-1} H(p_0)$
\end{corollary}
\begin{proof}
With $c_1 := n\gamma_1$ and $\eta$ as given, let $\hat a$ and $\hat b$ be chosen as in the proof of the theorem above. Since the derivative estimate holds everywhere, the set $U = \{ x \in M : d_g(p_0, x) < \hat aH(p_0)^{-1}, H(x) > \hat b^{-1} H(p_0)\}$ is the entire ball $B_g(p_0, \hat a H(p_0)^{-1})$. If $\lambda_1 \geq \eta H$ holds everywhere in $U$, the theorem implies the first alternative holds. Otherwise, clearly the second alternative holds. 
\end{proof}

The next lemma is an auxiliary result for the proof of the Neck Detection Lemma below. Our goal is to show that whenever $\frac{\lambda_1}{H}(p_0, t_0)$ is sufficiently small then $p_0$ lies at the center of an $(\varepsilon_0, L)$-neck on time $t_0$. We begin by first showing smallness of $\frac{\lambda_1}{H}(p_0, t_0)$ implies the second fundamental form is close to the second fundamental form of a round cylinder in a small intrinsic ball around $p_0$ at time $t_0$. 

\begin{lemma}\label{neck_detection_one}
Suppose $\mathcal M$ satisfies $(\ast)$. Let $\rho := \frac{1}{4}\hat r$, where $\hat r$ is the positive constant appearing in Lemma \ref{curvature_control}. For every $\phi \in (0, \frac{1}{n})$, we can find $\eta := \eta(\phi, n, \gamma_1, \gamma_2) \in (0, \phi)$ with the property that if $(p_0, t_0)$ is a spacetime point that satisfies $\lambda_1(p_0, t_0) \leq \eta H(p_0, t_0)$, then for every $p \in B_{g(t_0)}\big(p_0, 2\rho \frac{n-1}{H(p_0, t_0)}\big)$ there holds: 
\begin{enumerate}
\item[(1)] $\lambda_1(p, t_0) \leq \phi  H(p_0, t_0)$ and $\lambda_n(p, t_0) - \lambda_2(p, t_0) \leq \phi H(p_0, t_0)$;
\item[(2)] $\sum_{k =1}^8 H(p_0, t_0)^{-k-1}|\nabla^k h(p, t_0)| \leq \phi $;
\item[(3)] $(1- \phi) H(p_0, t_0) \leq H(p, t_0) \leq (1 + \phi) H(p_0, t_0)$. 
\end{enumerate}
\end{lemma}
\begin{proof}
We argue by contradiction. Suppose the assertion is not true. Then for some $\phi \in (0, \frac{1}{n})$, there exists a sequence of flows $\mathcal M^{(j)}$ satisfying $(\ast)$ for some fixed values of $\gamma_1$ and $\gamma_2$ and a sequence of positive constants $\eta_j \to 0$ which are counterexamples to the assertion. Assume the flow $\mathcal M^{(j)}$ is given by an embedding $F_j : M_j \times (-\infty, 0] \to \R^{n+1}$ and let $g_j, h_j$, $H_j$, and $\lambda_{i,j}$ denote the metric, second fundamental form, mean curvature, and principal curvatures ($i = 1, \dots n$) on $\mathcal M^{(j)}$. Then for each $j$ there exist spacetime points $(p_j, t_j) \in M_j \times (-\infty, 0]$ such that $\lambda_{1,j}(p_j, t_j) \leq \eta_j H_j(p_j, t_j)$, but 
\[
\sup_{q \in B_j} \max \left\{ \frac{\lambda_{1,j}(q, t_0)}{H_j(p_j, t_j)}, \quad \frac{\lambda_{n, j}(q, t_j) - \lambda_{2,j}(q, t_j)}{H_j(p_j, t_j)}, \quad \sum_{k=1}^8 \frac{|\nabla^k h_j(q, t_j)|}{H_j(p_j, t_j)^{k+1}}, \quad \Big|\frac{H_j(q,t_j)}{H_j(p_j, t_j)} -1\Big|\right \} > \phi,
\]
where $B_j := B_{g_j(t_j)}\big(p_j, 2\rho \frac{n-1}{H_j(p_j, t_j)}\big)$.   

Let $Q_j = (n-1) H_j(p_j, t_j)^{-1}$. The desired inequalities are scale and translation invariant, so as in the proof of Lemma \ref{optimal_two_convexity}, for each $j$ we define a translated and rescaled solution $\tilde F_j$ by 
\[
\tilde F_j(p, \tau) = \frac{1}{Q_j} \big[ F_j\big(p, Q_j^2 \tau + t_j) - F_j(p_j, t_j) \big]. 
\]
Let $\tilde {\mathcal M}^{(j)}$ denote the resulting flow. Let $\tilde g_j, \tilde h_j$, $\tilde H_j$, and $\tilde \lambda_{i,j}$ denote the geometric quantities of $\tilde {\mathcal M}^{(j)}$. By construction, each flow $\tilde{\mathcal M}^{(j)}$ satisfies $\tilde H_j(p_j, 0) = n-1$ and $\tilde \lambda_{1,j}(p_j, 0)\leq (n-1) \eta_j$. For $r > 0$, let $\tilde P_j(r)$ denote the parabolic neighborhood $P(p_j, 0,  r, r^2)$  in $\tilde{\mathcal M}^{(j)}$. Lemmas \ref{curvature_control} and \ref{derivative_estimates} implies uniform estimates
\[
\sup_{\tilde P_j(4\rho)} \,\Big|\pdv{t}^l \nabla^k \tilde h_j\Big|\leq C(k, l, n, \gamma_1, \gamma_2). 
\]
As before, a subsequence of the solutions $\tilde F_j$ converge locally in the $C^{20}$-topology within the parabolic neighborhoods $\tilde P_j(4\rho)$ to a smooth limit flow $\hat {\mathcal M}$ defined in a small intrinsic parabolic neighborhood $P_{\infty}(4\rho) := P(p_{\infty}, 0, 4\rho, 16\rho^2)$. Let $g, h, H$ and $\lambda_i$ denote the geometric quantities on $\hat{\mathcal M}$. 

Now we analyze the limit flow. By assumption and Proposition \ref{optimal_two_convexity}, $\tilde \lambda_{1,j} \geq 0$, $\tilde \lambda_{1,j} + \tilde \lambda_{2,j} \geq \frac{1}{n-1}\tilde H_j$, and $\tilde \lambda_{1,j}(p_j, 0) \leq (n-1)\eta_j$, where $\eta_j \to 0$ as $j \to \infty$. Hence in the limit, we have $\lambda_1 \geq 0$, $\lambda_1 + \lambda_2 \geq \frac{1}{n-1} H$ and $\lambda_1(p_{\infty}, 0) = 0$. As in the proof of Proposition \ref{optimal_two_convexity}, the maximum principle implies $\lambda_1 \equiv 0$ in $P_{\infty}(4\rho)$. Consequently, the limit flow must split a line and in view of the estimate $\lambda_1 + \lambda_2 \geq \frac{1}{n-1}H$, this gives $\lambda_2 = \cdots = \lambda_n = \frac{1}{n-1} H$. The Schur lemma implies the cross-section is a piece of an $(n-1)$-sphere and hence that $P_{\infty}(4\rho)$ is a small parabolic neighborhood in an evolving family of shrinking cylinders. In particular, $\lambda_1 \equiv 0$, $\lambda_n \equiv \lambda_2 $, $\sum_{k =1}^8|\nabla^k h| \equiv 0$, and $H \equiv n-1$ in $P_{\infty}(4\rho)$. By convergence of the $\tilde F_j$ in the $C^{20}$-topology, this implies 
\[
\sup_{q \in B_j} \max \left\{ \frac{\lambda_{1,j}(q, t_0)}{H_j(p_j, t_j)}, \quad \frac{\lambda_{n, j}(q, t_j) - \lambda_{2,j}(q, t_j)}{H_j(p_j, t_j)}, \quad \sum_{k=1}^8 \frac{|\nabla^k h_j(q, t_j)|}{H_j(p_j, t_j)^{k+1}}, \quad \Big|\frac{H_j(q,t_j)}{H_j(p_j, t_j)} -1\Big|\right \} \to 0
\]
along a subsequence as $j \to \infty$, in contradiction with our previous assumption for $j$ sufficiently large.
\end{proof}

By iterating the lemma above finitely many times, we can prove our version of the Neck Detection Lemma (cf. Lemma 7.4 in \cite{HS09}, Lemma 4.2 in \cite{Ngu18}, and Theorems 2.14 and 2.15 in \cite{BH16}). 

\begin{lemma}[Neck Detection]\label{neck_detection}
Suppose $\mathcal M$ satisfies $(\ast)$. Given $\varepsilon_0 \in (0, \frac{1}{n})$ and $L \geq 100$, there exists $\eta_0:= \eta_0(\varepsilon_0, L, n, \gamma_1, \gamma_2) \in (0, \varepsilon_0)$ with the property that if $(p_0,t_0)$ is a spacetime point with $\lambda_1(p_0,t_0) \leq  \eta_0 H(p_0,t_0)$, then $p_0 \in M$ lies at the center of an $(\varepsilon_0, L)$-neck at time $t_0$. 
\end{lemma}
\begin{proof}
Let $\varepsilon_0$ and $L$ be given. After a rescaling, we can assume without loss of generality that $H(p_0, t_0) = n-1$. Choose $\varepsilon_1 \in (0, \varepsilon_0)$ so that the conclusion of Proposition \ref{curvature_necks} holds. Recall the choice of $\varepsilon_1$ depends only upon $n$, $\varepsilon_0$, and $L$. By Proposition \ref{curvature_necks}, to show that $p_0$ lies at the center of a $(\varepsilon_0, L)$-neck at time times $t_0$, it suffices to show we can choose $\eta_0$ so that the following two properties hold:
\begin{enumerate}
\item[(a)] $\lambda_1(p_0, t_0) \leq \varepsilon_1  H(p_0, t_0)$ and $\lambda_n(p_0, t_0) - \lambda_2(p_0, t_0) \leq \varepsilon_1 H(p_0, t_0)$.
\item[(b)] For every $p \in B_{g(t_0)}(p_0, L+10)$, $\sum_{k =1}^8 H(p_0, t_0)^{-k-1}|\nabla^k h(p, t_0)| \leq \varepsilon_1$. 
\end{enumerate}

Recall from the previous lemma that $\rho := \frac{1}{4} \hat r$. Choose $N$ sufficiently large such that 
\[
1 + \sum_{m=2}^{N} \frac{m-1}{m} > \frac{L + 10}{\rho}. 
\]
We will now determine $\eta_0$. Begin by choosing $\eta_N \in (0, \varepsilon_1)$ such that the conclusions of Lemma \ref{neck_detection_one} hold with $\phi = \varepsilon_1$. For $m = 1, \dots, N-1$, having chosen $\eta_{m+1}$, first choose $\tilde \eta_{m} \in (0, \eta_{m+1} )$ so that the conclusions of Lemma \ref{neck_detection_one} hold with $\phi = \eta_{m+1}$ and then let $\eta_m := \min\{\frac{\tilde \eta_m}{1 + \tilde \eta_m},  \frac{1}{N -m}\}$. Finally, by one further application of the Lemma \ref{neck_detection_one}, choose $\tilde \eta_0 \in (0, \eta_1)$ so that Lemma \ref{neck_detection_one} holds with $\phi = \eta_1$ and define $\eta_0:= \min\{\tilde \eta_0,  \frac{\varepsilon_1}{n-1}\}$. It is clear that $\eta_0$ has been chosen in a way that depends only upon $\varepsilon_0, L, n, \gamma_1$ and $\gamma_2$. 

Now let us examine the consequences of our choices. Assume $\lambda_1(p_0, t_0) \leq \eta_0 H(p_0, t_0)$. First since $\eta_0 \leq \varepsilon_1$, we have $\lambda_1(p_0, t_0) \leq \varepsilon_1 H(p_0, t_0)$. Since $\mathcal M$ is $\frac{1}{n-1}$-two-convex, we have $\lambda_2(p_0, t_0) \geq (\frac{1}{n-1} - \eta_0) H(p_0, t_0)$ which implies 
\[
\lambda_n(p_0, t_0) - \lambda_2(p_0, t_0)  \leq H(p_0, t_0) -(n-1)\lambda_2(p_0, t_0) \leq  (n-1) \eta_0H(p_0, t_0) \leq \varepsilon_1 H(p_0, t_0). 
\]
This shows that (a) holds for our choice of $\eta_0$. As for (b): Define $\rho_1 = \rho$. For $m = 2, \dots, N$, define 
\[
\rho_m := \rho \Big(1 + \sum_{j = 1}^{m-1} \frac{N-j}{N+1-j}\Big) = \rho\Big(1 + \sum_{j = N + 2- m }^N \frac{j -1}{j}\Big)
\] 
Note that $\rho_{m+1} = \rho_m + \rho \frac{N-m}{N+1 -m} $ and $\rho_N > L + 10$ by our choice of $N$. Let us say condition $(\Gamma_m)$ holds if for all $p \in B_{g(t_0)}(p_0, \rho_m)$, we have
\begin{enumerate}
\item[(i)] $\lambda_1(p, t_0) \leq \eta_m H(p_0, t_0)$ and $\lambda_n(p, t_0) - \lambda_2(p, t_0) \leq \eta_m H(p_0, t_0)$;
\item[(ii)] $\sum_{k =1}^8 H(p_0, t_0)^{-k-1}|\nabla^k h(p, t_0)| \leq \eta_m$;
\item[(iii)] $(1- \eta_m)H(p_0, t_0) \leq H(p, t_0) \leq (1+ \eta_m) H(p_0, t_0)$. 
\end{enumerate}
Our choice of $\eta_0 \leq \tilde \eta_0$ and Lemma \ref{neck_detection_one} implies that $(\Gamma_1)$ holds. Suppose that $1 \leq m < N$ and $(\Gamma_m)$ holds. Now consider an arbitrary point $p \in B_{g(t_0)}(p_0, \rho_m)$. First, because $\eta_m \leq \frac{\tilde \eta_m}{1 + \tilde \eta_m}$, we have 
\[
\lambda_1(p, t_0) \leq  \frac{\tilde \eta_m}{1 + \tilde \eta_m} H(p_0, t_0) \leq  \frac{\tilde \eta_m}{1 + \tilde \eta_m} \frac{1}{1 - \eta_m} H(p, t_0) \leq \tilde \eta_m H(p, t_0). 
\]
Second, because $\eta_m \leq \frac{1}{N -m}$, we have 
\[
\frac{n-1}{H(p, t_0)} \geq \frac{1}{1 + \eta_m} \frac{n-1}{H(p_0, t_0)} \geq \frac{N -m}{N +1 -m}.
\]
We chose $\tilde \eta_m$ to satisfy the conclusions of Lemma \ref{neck_detection_one} with $\phi = \eta_{m+1}$. Given that $\lambda_1(p, t_0) \leq \tilde \eta_m H(p, t_0)$ and $\frac{n-1}{H(p, t_0)} \geq \frac{N-m}{N +1 - m}$, we conclude that conditions (i), (ii), and (iii) above hold in the ball $B_{g(t_0)}(p, 2\rho \frac{N-m}{N+1 -m})$ with $\eta_m$ replaced by $\eta_{m+1}$. Since $p \in B_{g(t_0)}(p_0, \rho_m)$ is arbitrary and $\rho_m + \rho \frac{N-m}{N+1 - m} = \rho_{m+1}$, this implies $(\Gamma_{m+1})$ holds. By finite induction $(\Gamma_N)$ holds and this implies (b) holds for $\eta_0$, completing the proof of the lemma. 
\end{proof}


\section{Neck and Cap Decomposition: Noncompact Case}

In this section, we will assume $M$ is noncompact. We begin this section by recalling some useful analysis on the convexity of necks from Proposition 7.18, Lemma 7.19, and the surrounding discussion in \cite{HS09}. The following lemma highlights a difference between our setting (a noncompact, convex hypersurface) and the setting of Huisken and Sinestrari (a closed, two-convex hypersurface). In the setting of Huiksen-Sinestari, the axes of different necks need not align if the (intrinsic) distance between the neck regions is large compared to the curvature scales of the necks. In our setting, we can show every neck must have approximately the same axis. 

\begin{lemma}\label{approx_axis}
Suppose $F : M \to \R^{n+1}$ is an embedding of a noncompact, complete, convex hypersurface. There exists a unit vector $\omega \in S^n$ with the following two properties: 
\begin{enumerate}
\item[(1)] $\langle \nu(q), \omega \rangle \geq 0$ for every $q \in M$. 
\item[(2)] There exists a constant $C := C(n)$ such that if $N \subset M$ is an $(\varepsilon, L)$-neck, then $\langle \nu(q), \omega \rangle \leq C\varepsilon$ for every $q \in N$. 
\end{enumerate}
\end{lemma} 
\begin{proof}
We begin by constructing $\omega$. Let $\Omega$ denote the convex interior of $F(M)$. Choose any sequence of points $x_k \in \Omega$ such that $|x_k| \to \infty$. We can find a subsequence of these points such that $-\frac{x_k}{|x_k|}$ converges to a limit $\omega \in S^n$. A consequence of convexity is that $\langle \nu(q), \omega \rangle \geq 0$ for all $q \in M$. To see this, consider any point $x \in \bar \Omega$ and fix some $s \geq 0$. For $k$ sufficiently large, $s_k := s|x_k|^{-1} \in [0, 1]$. By convexity, $(1 - s_k)x + s_k x_k \in \bar \Omega$ for $k$ sufficiently large. As $k \to \infty$, $(1 - s_k)x + s_k x_k$ converges to $x - s \omega$, and therefore $x - s\omega \in \bar \Omega$ for all $s \geq 0$. A convex hypersurface always lies above its tangent planes; so if $q \in M$ and $x \in \bar \Omega$, then $\langle x - F(q), -\nu(q) \rangle \geq 0$. Setting $x = F(q) - \omega$ gives $\langle \nu(q), \omega \rangle \geq 0$, as claimed. 

Now suppose $N \subset M$ is an $(\varepsilon, L)$-neck. By definition $N$ can be expressed as a small graph over an embedded round cylinder $\Sigma$ of length $2L$ in $\R^{n+1}$. Let $\omega_0$ be the unit vector parallel to the axis of $\Sigma$ such that $\langle \omega, \omega_0 \rangle \geq 0$. It follows from the definition of a neck that: 
\begin{enumerate}
\item[(i)] $|\langle \nu(q), \omega_0 \rangle |\leq C \varepsilon$ for every $q \in N$.
\item[(ii)] If $e$ is a unit vector orthogonal to $\omega_0$, then there exists a point $q \in N$ where $|\nu(q) - e| \leq C \varepsilon$. 
\end{enumerate}
If $\omega = \omega_0$, then we are done. Otherwise, consider $v = \omega - \langle \omega, \omega_0\rangle \omega_0 \neq 0$. By (ii), we can find a point $q \in N$ where $|\nu(q) + \frac{v}{|v|}| \leq C \varepsilon$. Then at $q$ we have 
\[
0 \leq \langle \nu(q), \omega \rangle = - \big\langle \frac{v}{|v|}, \omega \big\rangle + \big\langle \nu(q) + \frac{v}{|v|}, \omega \big\rangle \leq -\sqrt{1 - \langle \omega, \omega_0 \rangle^2} +  C\varepsilon.
\]
This gives $|\omega - \omega_0| \leq C\varepsilon$ and therefore, with $(i)$, we conclude $\langle \nu, \omega \rangle \leq C\varepsilon$ everywhere on $N$. 
\end{proof}

Now suppose $F : M \to \R^{n+1}$ is an embedding of a noncompact, complete, convex hypersurface that satisfies the gradient estimate $|\nabla h | \leq \gamma_1 H$. Suppose $p \in M$ lies at the center of an $(\varepsilon, L)$-neck $N$. By the lemma above, we can find $\omega \in S^n$ such that $\langle \nu, \omega \rangle \geq 0$ everywhere on $M$ and $\langle \nu, \omega \rangle \leq C\varepsilon$ everywhere on $N$. Let $y$ denote the height along the axis defined by $\omega$, normalized so that $p$ is contained in the hyperplane $y = 0$. Let $\Sigma_0$ denote the intersection of $N$ with the level set $y = 0$. The definition of a neck implies $\Sigma_0$ is compact and very close to a round $(n-1)$-dimensional sphere. We will call the $y$-direction vertical and all other directions, orthogonal to $\omega$, horizontal. 

By assumption $\omega$ is nearly tangent to $N$. As in \cite{HS09}, we consider integral curves of the height function $y$. For each $q \in \Sigma_0$, let $\gamma(\tau) := \gamma(\tau, q)$ be a solution to the ODE 
\[
\begin{cases} \dot \gamma = \frac{\omega^\top(\gamma)}{|\omega^\top(\gamma)|^2} & \tau \geq 0, \\ \gamma(0) = q, & \\ \end{cases}
\]
where $\omega^\top$ denotes the projection of $\omega$ to the tangent space of $M$. The curves are defined so that $\frac{d}{d\tau} y(\gamma(\tau)) = \langle \gamma'(\tau), \omega \rangle = 1$. So by our normalization $y(\Sigma_0)=0$, we have $y(\gamma(\tau)) = \tau$. Hence we can write $\gamma(y)$ in place of $\gamma(\tau)$. We will consider these curves for as long as they are well-defined, including after they leave the neck region $N$. As $\Sigma_0$ is compact, every curve is defined for $|y|$ small.  Let $y_{\min} \in [-\infty, 0)$ and $y_{\max} \in (0, \infty]$ be the minimal and maximal heights such that for every $q \in \Sigma_0$, the curve $\gamma(\cdot, q)$ is defined for $y \in (y_{\min}, y_{\max})$. It is possible for either of $y_{\min}$ or $y_{\max}$ to be infinite because our hypersurface is noncompact. However, we will see that the assumptions on $\omega$ will ensure $y_{\max} < \infty$ and $y_{\min} = -\infty$. For each $y \in (y_{\min}, y_{\max})$, let $\Sigma_y = \{\gamma(y, q) : q \in \Sigma_0\}$. Since our hypersurface is convex, the $\Sigma_y$ are just level sets of the height function. We will say the surfaces $\Sigma_y$ are shrinking if the projection of $\Sigma_{y_2}$ to a fixed hyperplane $y = y'$ is contained in the domain enclosed by the projection of $\Sigma_{y_1}$ to the hyperplane $y =y'$ for any $y_2 \geq y_1$. 

Now we give a lemma concerning the behavior integral curves to the height function and the surfaces they define. The lemma is a combination of Proposition 7.18 and Lemma 7.19 in \cite{HS09}. A slight difference is that our gradient estimate holds at all curvature scales and our hypersurface is noncompact. 

\begin{lemma}\label{curves}
Suppose $F : M \to \R^{n+1}$ is an embedding of a noncompact, complete, convex hypersurface satisfying the gradient estimate $|\nabla h | \leq \gamma_1 H$. Let $N \subset M$ be an $(\varepsilon, L)$-neck. Let $\omega$ satisfy the conclusions of Lemma \ref{approx_axis}. Under the hypotheses above, if $0 < \varepsilon < \varepsilon(n)$ is sufficiently small, then there holds:
\begin{enumerate}
\item[(1)] For every $q \in \Sigma_0$, the curve $\gamma(\cdot, q)$ is well-defined for as long as it is contained in the neck $N$. 
\item[(2)] Along a trajectory $\gamma$ (in the direction of $\omega$), we have 
\[
\frac{d}{dy} \langle \nu, \omega \rangle \geq \lambda_1 > 0.
\]
\item[(3)] Suppose $\langle \nu(q), \omega \rangle > 0$ for all $q \in \Sigma_{0}$. Then the surfaces $\Sigma_y$ are shrinking for all $y \in [0, y_{\max})$ and $y_{\max} < \infty$. Moreover, there exists a positive constant $\hat \theta := \hat \theta(n, \gamma_1)$ such that $H(\gamma(y,q)) \geq \hat \theta^{-1} H(p)$ for all $q \in \Sigma_0$ and $y\in [0 , y_{\max})$. 
\end{enumerate}
\end{lemma}
\begin{proof}
Statement $(1)$ is clear from the inequality $0 \leq \langle \nu, \omega \rangle \leq C\varepsilon$ on $N$ and $(2)$ is a computation as in the proof of Proposition 7.18 in \cite{HS09}. Statement $(4)$ is essentially Lemma 7.19 in \cite{HS09}, but we verify it here. Because $\Sigma_{0}$ is compact, we may assume $\langle \nu(q), \omega \rangle \geq \varepsilon' > 0$ for all $q \in \Sigma_{0}$. By $(2)$, this implies $\langle \nu(\gamma(y,q)), \omega \rangle \geq \varepsilon'$ for all $q \in \Sigma_0$ and all $y \in [0, y_{\max})$. For $\bar y \in [0, y_{\max})$, consider the projection of the surface $\Sigma_{\bar y}$ to the hyperplane $y = 0$. The outward-pointing normal direction of the projected surface is $\nu - \langle \nu, \omega \rangle \omega$. We compute for any $q \in \Sigma_0$,
\[
\Big\langle \dot \gamma, \frac{\nu - \langle \nu, \omega \rangle \omega}{|\nu - \langle \nu, \omega \rangle \omega |} \Big\rangle =  \Big\langle \frac{\omega^\top}{|\omega^\top|^2},  \frac{\nu - \langle \nu, \omega \rangle \omega}{\sqrt{1 - \langle \nu, \omega \rangle^2}} \Big\rangle= - \frac{\langle \nu, \omega \rangle}{\sqrt{1 - \langle \nu, \omega \rangle^2}} \leq - \frac{\varepsilon'}{\sqrt{1 - (\varepsilon')^2}}. 
\]
This shows the horizontal component of $\dot \gamma(\bar y, q)$ points towards the interior of $\Sigma_{\bar y}$ and has norm at least $\varepsilon'/\sqrt{1 - (\varepsilon')^2}$. This means the surfaces $\Sigma_{\bar y}$ are shrinking at definite rate for $y \geq 0$ and hence $y_{\max} < \infty$. The proof of the second statement in $(3)$ follows directly from the proof of Lemma 7.19 in \cite{HS09}. Because $\omega$ is approximately the axis of our neck, $\Sigma_0$ is very close to a standard $(n-1)$-sphere of radius $\frac{n-1}{H(p)}$. Supposing $\varepsilon$ is sufficiently small, this implies $H(q) \geq \frac{1}{2} H(p)$ for all $q \in \Sigma_0$ and that there exists an $(n-1)$-sphere of radius $R = 2\frac{n-1}{H(p)}$ that encloses $\Sigma_{0}$ in the hyperplane $y = 0$. Recall, as in the proof of Lemma \ref{curvature_control}, for any points $q, q' \in M$, we have 
\[
H(q) \geq \frac{1}{H(q')^{-1} + n \gamma_1 d_g(q, q')}. 
\]
If $\bar y \in [0, y_{\max})$ satisfies $\bar y < R$, then for any $q \in \Sigma_{\bar y}$, it is clear there exists $q' \in \Sigma_{0}$ such that $d_g(q, q') \leq 2R$ (the extrinsic distance between points in $\Sigma_0$ and $q \in \Sigma_{\bar y}$ is bounded by $R$ vertically and $R$ horizontally; since our hypersurface is strictly convex, the intrinsic distance is similarly bounded). Since $H(q') \geq \frac{1}{2}H(p)$, we have 
\[
H(q) \geq \frac{1}{H(q')^{-1} + n \gamma_1 d_g(q,q')} \geq \frac{1}{2H(p)^{-1}+ n \gamma_1 2R} = \frac{H(p)}{2 + 4n(n-1) \gamma_1}. 
\]
On the other hand, if $\bar y \geq R$, then we can find $y'$ such that $\bar y \in [y', y' + R] \subset [0, y_{\max})$. We can construct a suitable portion of cone with spherical cross-section, axis $\omega$, and bases in the hyperplanes $y = y'$ and $y = y' + R$ of radius $R_1, R_2 \leq R$ respectively. For suitable choices of $R_1, R_2$, we can arrange that the cone touches a point $q' \in \cup_{y \in (y', y' + R)} \Sigma_{y}$ from the outside. This is possible by convexity of $\cup_{y \in (y', y' + R)} \Sigma_{y}$ and because the surfaces are shrinking. Now $H(q') \geq \frac{n-1}{R} = \frac{1}{2}H(p)$ by comparison to the cone. If $q \in \Sigma_{\bar y}$, noting that the intrinsic diameter of $\Sigma_{\bar y}$ is bounded by $\pi R$, then $d_g(q, q') \leq (2 + \pi) R$. Thus the above argument applies. We can take $\hat \theta := (2 + 2(2+ \pi)n(n-1) \gamma_1)^{-1}$ to complete the proof. 
\end{proof}

In the next step, we prove our ancient solution has a convex cap outside of which every point lies at the center of a neck. For the mean curvature flow of two-convex hypersurfaces, the following key result is often called the Neck Continuation Theorem. See Theorem 8.2 in \cite{HS09} and also Theorem 3.2 in \cite{BH16}. The proof of our version of the Neck Continuation Theorem is modeled on the proofs given by Huisken, Sinestrari, and Brendle. Of course our argument is also a bit simpler in that we do not need to consider if regions have been previously affected by surgery. Our phrasing of the Neck Continuation Theorem is inspired by similar statements used by Perelman in his study of $\kappa$-solutions in the Ricci flow.

Before the theorem, let us point out that if $N$ is an $(\varepsilon, L)$-neck in a noncompact, complete, connected, strictly convex hypersurface $M$, then $M \setminus N$ consists of two connected components, one bounded and the other unbounded. If both components were unbounded, then $M$ would split a line, thereby contradicting strict convexity. 


\begin{theorem}\label{neck_cap_noncompact}
Suppose $\mathcal M$ is noncompact and satisfies $(\ast)$. Given $0 <\varepsilon_0 < \varepsilon(n)$ small and $L \geq 100$, there exist constants $\varepsilon_1 \in (0, \varepsilon_0)$, and $C_0 < \infty$, depending only upon $\varepsilon_0$, $L$, $n$, $\gamma_1$ and $\gamma_2$, so that the following holds. Fix any time $t \in (-\infty, 0]$. Suppose that $p \in M$ is a point which lies at the center of an $(\varepsilon_1, L)$-neck $N$ at time $t$, and suppose further that $p$ does not lie at the center of an $(\frac{\varepsilon_1}{2}, 2L)$-neck at time $t$. Let $D$ denote the bounded connected component of $M \setminus N$, and let $\tilde{D}$ denote the unbounded connected component of $M \setminus N$. Then:
\begin{enumerate}
\item[(1)] Every point $q \in \tilde{D} \cup N$ lies at the center of an $(\varepsilon_0, L)$-neck at time $t$. 
\item[(2)] $D$ is diffeomorphic to $B^n$. 
\item[(3)] $\d D \subset \d N$ is a cross-sectional sphere of an $(\varepsilon_0, L)$-neck. 
\item[(4)] The (intrinsic) diameter of $D$ is bounded by $C_0 H(p,t)^{-1}$. 
\item[(5)] Every point $q \in D$ satisfies $C_0^{-1} H(p,t) \leq H(q,t) \leq C_0 H(p,t)$ and $\lambda_1(q,t) \geq C_0^{-1} H(q,t)$. 
\end{enumerate}
\end{theorem}
\begin{proof}
Fix a time $t_0$ and for simplicity let us suppress $t_0$ in our notation. For any point $q \in M$, let $r_q := (n-1)H(q)^{-1}$ denote the mean curvature scale. As usual, let $B_g(q, r)$ denote an intrinsic ball of radius $r$ around $q$. Let $\varepsilon_0 > 0$, suitably small, and $L \geq 100$ be given. We can assume that on any $(\varepsilon_0, L)$-neck, the mean curvature satisifes $\frac{9}{10} \leq \frac{H(q_1)}{H(q_2)} \leq \frac{10}{9}$ for any pair of points $q_1, q_2$ on the neck. We will determine the constants $\varepsilon_1 \in (0, \varepsilon_0)$ and $C_0$ in two steps. By Lemma \ref{approx_axis}, we can fix unit vector $\omega \in S^n$ with the property that $\langle \nu, \omega \rangle \geq 0$ everywhere on $M$ and $0 \leq \langle \nu, \omega \rangle \leq C \varepsilon_1$ on any $(\varepsilon_1, L)$-neck $N \subset M$. \\

\textit{Step 1:} For any $0 < \varepsilon_1 < \varepsilon(n) $ sufficiently small, there exists $C_0:= C_0(\varepsilon_1, n, L, \gamma_1, \gamma_2) < \infty$ such that if $p$ and $N$ satisfy the assumptions of the theorem, then parts (2), (4), and (5) of the theorem hold. \\

We assume $\varepsilon_1 < \varepsilon(n)$ is sufficiently small so that the conclusions of Lemma \ref{curves} hold and so that $\frac{9}{10} \leq \frac{H(q_1)}{H(q_2)} \leq \frac{10}{9}$ for any two points on an $(\varepsilon_1, L)$-neck. Given such $\varepsilon_1$ and $p$ lying on $N$ as in the theorem, we can find $\eta_1 > 0$ such that $\lambda_1 > \eta_1 H$ everywhere on the neck $N$. This follows from the assumption that $p$ does not lie at the center of an $(\frac{\varepsilon_1}{2}, 2L)$-neck. The proof uses two applications of the Neck Detection Lemma and the definition of a neck. 
\begin{enumerate}
\item[$\bullet$] First, by the Neck Detection Lemma, we can find $\hat \eta:=\hat \eta(\varepsilon_1, n, L, \gamma_1, \gamma_2) \in (0, \frac{\varepsilon_1}{2})$ such that if $\lambda_1(q) \leq \hat \eta H(q)$, then $q$ lies at the center of an $(\frac{\varepsilon_1}{2}, 2L)$-neck.  
\item[$\bullet$] Next, using the definition of a neck, it is possible to choose $\hat \varepsilon:= \hat \varepsilon(\hat \eta, n) \in (0, \hat \eta)$ sufficiently small such if $q$ lies at the center of an $(\hat \varepsilon, 3L)$-neck $\hat N$, then $\lambda_1 \leq \hat \eta H$ everywhere on $\hat N$ and $B_g(q, 2L r_q) \subset \hat N$. This possible since an exact cylinder of length $6L$ contains an intrinsic ball of radius $2L$ around any point on its central sphere and on a cylinder $\lambda_1 \equiv 0$. 
\item[$\bullet$] Finally, by another application of the Neck Detection Lemma, choose $\eta_1:= \eta_1(\hat \varepsilon, n, L, \gamma_1, \gamma_2) \in (0, \hat \varepsilon)$ such that if $\lambda_1(q) \leq \eta_1 H(q)$, then $q$ lies at the center of an $(\hat \varepsilon, 3L)$-neck $\hat N$. 
\end{enumerate}
By our assumption on $\varepsilon_1$, for every $q \in N$, we have $d_g(p,q) < \frac{3}{2}Lr_p < 2Lr_q$, which implies $p \in B_g(q, 2Lr_q)$. Because $p$ does not lie at the center of an $(\frac{\varepsilon_1}{2}, 2L)$-neck, we have $\lambda_1(p) > \hat \eta H(p)$ and thus $\lambda_1(q) > \eta_1 H(q)$ for all $q \in N$, as claimed. 

We now show how these properties above imply parts (2), (4), and (5) of the theorem. Our primary tool, as in \cite{HS09}, is to analyze the integral curves of the height function. Define $y : M \to \R$ by $y(q) = \langle F(q)- F(p), \omega \rangle$. Our normalization ensures $y(p) = 0$. Let $\Sigma_0 \subset N$ denote the level set $y = 0$. By Lemma \ref{approx_axis}, $\Sigma_0$ is $O(\varepsilon_1)$-close to an spherical cross-section of $N$. For every $q \in \Sigma_0$, define $\gamma(y, q)$ to be the integral curve of $\frac{\omega^\top}{|\omega^\top|^2}$ passing through $q$. As in Lemma \ref{curves}, these curves are defined for $y \in (y_{\min}, y_{\max})$, with $y_{\min} < 0 < y_{\max}$, and we know they are well-defined at least as long as they are in $N$. For each $y$, let $\Sigma_y = \{ \gamma(y, p) : p \in \Sigma_0\}$ denote the smooth level sets of the height function. For $y_{\min} \leq y_1 < y_2 \leq y_{\max}$, let $\Sigma(y_1, y_2) = \bigcup_{y_1 < y < y_2 } \Sigma_y$. 

Because $\omega$ is an approximate axis of $N$ and the neck has intrinsic length approximately equal to $2Lr_p$ and $L \geq 100$, we must have $\Sigma(-2r_p, 0) \subset N$. Since $\frac{d}{dy} \langle \nu, \omega \rangle \geq \lambda_1$ and $\lambda_1 \geq \eta_1 H \geq \frac{1}{2}\eta_1 H(p)$ on $N$, for any $q \in \Sigma_0$ we have
\begin{align*}
\langle \nu(q), \omega \rangle &= \langle \nu(\gamma(-2r_p,q), 0), \omega \rangle + \int_{-2r_p}^0 \frac{d}{dy} \langle \nu, \omega \rangle \, dy \\
& \geq r_p\eta_1H(p) = (n-1)\eta_1.
\end{align*}
The strict positivity of $\langle \nu, \omega \rangle$ implies our neck must close up. Indeed, by part $(3)$ of Lemma \ref{curves}, $y_{\max} < \infty$ and $H(\gamma(y, q)) > \hat \theta^{-1} H(p)$ for all $q \in \Sigma_0$ and $y \in [0, y_{\max})$. Here $\hat \theta$ depends only on $n$ and $\gamma_1$. 

By definition of $y_{\max}$, there exists a point $q \in \Sigma_0$ such that $\nu(\gamma(y, q)) \to \omega$ as $y \to y_{\max}$. Since $y_{\max} < \infty$, we can follow the remainder of the argument given in $\cite{HS09}$. Since our hypersurface is strictly convex, the arguments in \cite{HS09} already show that our surface closes up in a convex cap. Our goal is to show the cap is uniformly convex, depending only upon the given constants. We will show there exists $\eta_2 := \eta_2(\eta_1, n, L, \gamma_1, \gamma_2) \in (0, \eta_1)$ such that the following four properties hold for all $y \in [0, y_{\max})$:
\[
 |\langle \nu, \omega \rangle | < 1, \quad \lambda_1 > \eta_2 H , \quad H > \hat \theta^{-1} H(p), \quad \langle \nu, \omega \rangle > \eta_1.
\] 
If $\eta_2 \leq \frac{1}{2} \eta_1$, then these four properties hold for $y$ sufficiently close to zero. If they do not hold until $y_{\max}$, let $\tilde y \in (0, y_{\max})$ be the first value of $y$ for which one of these properties fails. The first property must hold for $y < y_{\max}$ by definition. Also the fourth property holds until $y_{\max}$ because $\langle \nu, \omega \rangle > \eta_1$ on $\Sigma_{0}$, $\frac{d}{dy}\langle \nu, \omega \rangle \geq \lambda_1$, and by assumption $\lambda_1 > 0$ everywhere. In particular, $\langle \nu, \omega \rangle > \eta_1$ on $\Sigma_{\tilde y}$. We have already seen by Lemma \ref{curves} that the third property must hold. So if a property fails at $\tilde y$, it must be the second one. If it fails, then there exists $\tilde p \in \Sigma_{\tilde y}$ such that $\lambda_1(\tilde p) \leq \eta_2 H(\tilde p)$. By the Neck Detection Lemma and Lemma \ref{approx_axis}, if we take $\eta_2$ sufficiently small, then $\tilde p$ lies at the center of a very fine neck $\tilde N$ and satisfies $\langle \nu(\tilde p), \omega \rangle < \eta_1$. This contradicts the fourth inequality above. 

With the existence of $\eta_2$ established and the above properties verified, elementary arguments (see the end of the proof of Theorem 8.2 in \cite{HS09}) imply all integral curves converge to the same critical point of the height function. Thus the region $M \cap \{y > 0\}$ is a uniformly convex cap diffeomorphic to $B^n$. Since $\omega$ is $C\varepsilon_1$-close to the axis of the cylinder over which $N$ is a graph, this implies $D$ is diffeomorphic to $B^n$, which is (2). So far, we have shown that for all $q \in D$, $\lambda_1(q) > \eta_2 H(q)$ and $H(q) > \hat \theta^{-1} H(p)$. We now prove a diameter bound and an upper mean curvature bound for points in $D$. Let $q_0$ be an arbitrary point on $D$. $M$ is noncompact, so by Corollary \ref{nearby_neck}, we can find constants $\hat a$ and $\hat b$, depending upon $\eta_2$, $n$, and $\gamma_1$, but independent of $q_0$, such that every point $q$ in the intrinsic ball $B:= B_g(q_0, \hat a H(q_0)^{-1})$ satisfies $H(q) \geq \hat b^{-1} H(q_0)$ and there exists a point $q_1 \in B_g(q_0, \hat a H(q_0)^{-1})$ with $\lambda_1(q_1) \leq \eta_2 H(q_1)$. Since $\lambda_1 > \eta_2H$ on the connected region $D \cup N$, it follows that the point $q_1$ is contained in $\tilde D$ and therefore the ball $B$ has nonempty intersection with $N$. The mean curvature of points in $N$ is upper bounded by $2H(p)$ and the mean curvature of points in $B$ is lower bounded by $\hat b^{-1} H(q_0)$. Putting these together gives the upper bound $H(q_0) \leq 2\hat bH(p)$. The intrinsic distance from $q_0$ to $N$ is bounded by $\hat a H(q_0)^{-1} \leq \hat a \hat \theta H(p)^{-1}$ and the distance between any two points on the neck $N$ is bounded by $100nLH(p)^{-1}$. Therefore $d_g(q_0, p) \leq (\hat a \hat \theta + 100nL) H(p)^{-1}$. This implies the diameter bound for $D$. Choosing 
\[
C_0 := \max\{\eta_2^{-1}, \hat \theta, 2\hat b, 2(\hat a\hat \theta + 100nL)\}
\]
completes the proofs of parts (4) and (5) of the theorem.  \\

\textit{Step 2:} There exists $\varepsilon_1 \in (0, \varepsilon_0)$, depending upon $n$, $L$, $\varepsilon_0$, $\gamma_1$, and $\gamma_2$, such that if $p$ and $N$ satisfy the assumptions of the theorem, then parts (1) and (3) of the theorem hold.\\
 
Since $\varepsilon_1 \leq \varepsilon_0$, part (3) of the theorem holds. Via the Neck Detection Lemma, we first choose $\eta_0:=\eta_0(\varepsilon_0, n, L, \gamma_1, \gamma_2) \in (0, \varepsilon_0)$ so that if $\lambda_1(q) \leq \eta_0 H(q)$, then $q$ lies on an $(\varepsilon_0, L)$-neck. Taking $\hat \varepsilon_0:= \hat \varepsilon_0(\eta_0, n) \in (0, \eta_0)$ sufficient small and using the Neck Detection Lemma once more, we can find $\hat \eta_0:=\hat \eta_0(\hat \varepsilon_0, n, L, \gamma_1, \gamma_2) \in (0, \eta_0)$ so that if $\lambda_1(q) \leq \hat \eta_0 H(q)$, then $q$ lies at the center of an $(\hat \varepsilon_0, L)$-neck with the property that $\lambda_1 \leq \frac{1}{2}\eta_0 H$ everywhere on the neck. We can now fix our choice of $\varepsilon_1$. Recall that we have $\langle \nu, \omega \rangle \geq 0$ everywhere on $M$ and $\langle \nu, \omega \rangle \leq C\varepsilon_1$ on the neck $N$. We assume $\varepsilon_1$ is sufficiently small so as to satisfy the following two inequalities: 
\[
\sup_{q \in N} \langle \nu(q), \omega \rangle  \leq  \hat \eta_0 \quad \text{ and }\quad \sup_{q \in N} \frac{\lambda_1(q)}{H(q)} \leq \eta_0. 
\]
It is clear we have chosen $\varepsilon_1$ in a way that only depends upon the given constants. 

As in the previous step, we consider the surfaces $\Sigma_y$, but now for $y < 0$.  To flip our orientation, define $\tilde \omega := - \omega$ and let $z := -y$ so that $z_{\max} := y_{\min}$. Then $\langle \nu, \tilde \omega \rangle \leq 0$ everywhere on $M$ and $- \hat \eta_0 \leq \langle \nu(q), \tilde \omega \rangle$ for all $q \in \Sigma_0$. Consider the curves $\tilde \gamma(z, q) := \gamma(-z, q)$ for $z \in [0, z_{\max})$ and all $q \in \Sigma_0$. Let $\tilde \Sigma_z := \Sigma_{-z}$ and $\tilde \Sigma(z_1, z_2) = \Sigma(-z_2, -z_1)$. First, we observe $z_{\max} = \infty$. This is because $\frac{d}{dz} \langle \nu, \tilde \omega \rangle = \frac{d}{dy} \langle \nu, \omega \rangle \geq \lambda_1 > 0$ which shows $\langle \nu, \tilde \omega \rangle$ is increasing in $z$. On the other hand $\langle \nu, \tilde \omega \rangle \leq 0$. In other words, for all $q \in \tilde \Sigma(0, z_{\max})$, $- \hat \eta_0 \leq \langle \nu(\tilde \gamma(z, q)), \tilde \omega \rangle \leq 0$, so we can never encounter a critical point of the height function in the direction of increasing $z$.

By construction, every point in $N$ lies at the center of an $(\varepsilon_0, L)$-neck. For sake of contradiction, suppose that there exist points in $\tilde D$ that do not lie at the center of an $(\varepsilon_0, L)$-neck. Let $\tilde z \in [0, \infty)$ be maximal among heights such that $\frac{\lambda_1}{H} \leq \eta_0$ holds for every $z \in [0, \tilde z]$. We have assumed $\varepsilon_1$ is sufficiently small so that this inequality holds at every point on $N$, which implies $\tilde z \geq \frac{1}{2}Lr_p > 0$. By maximality of $\tilde z$, we can find $\tilde q \in \tilde \Sigma_{\tilde z}$ such that $\lambda_1 (\tilde q) = \eta_0 H(\tilde q)$. In other words, $\tilde q$ barely lies on an $(\varepsilon_0, L)$-neck $\tilde N$. As we argued in the previous step, $\tilde \Sigma (\tilde z - 2 r_{\tilde q}, \tilde z) \subset \tilde N$. In view of our choice of $\hat \eta_0$, we must have $\lambda_1 > \hat \eta_0 H \geq \frac{1}{2}\hat \eta_0 H(\tilde q)$ on $\tilde \Sigma(\tilde z - 2 r_{\tilde q}, \tilde z)$, because otherwise we would contradict $\lambda_1(\tilde q) = \eta_0 H(\tilde q)$. Following the argument of the previous step, for any $q \in \tilde \Sigma_{0}$ we find 
\begin{align*}
\langle \nu(\gamma(\tilde z,q)), \tilde \omega \rangle & = \langle \nu(\gamma( \tilde z - 2r_{\tilde q}, q)), \tilde \omega \rangle + \int_{\tilde z - 2 r_{\tilde q}}^{\tilde z} \frac{d}{dz} \langle \nu, \tilde \omega \rangle \, dz \\
&\geq - \hat \eta_0 + 2 r_{\tilde q} \frac{1}{2} \hat \eta_0 H(\tilde q)\\ 
& =  (n-2) \hat \eta_0  > 0. 
\end{align*}
This contradicts our previous observation that $\langle \nu, \tilde \omega \rangle \leq 0$. This completes the proof of this step and the theorem. 
\end{proof}

For our final result in this section, we show there exist points satisfying the hypothesis of the theorem above. 

\begin{lemma}\label{existence_neck}
Suppose $\mathcal M$ is noncompact and satisfies $(\ast)$. Suppose $L \geq 100$. If $0 < \varepsilon_1 < \varepsilon(n)$ is sufficiently small, then for every $t_0 \in (-\infty, 0]$, we can find a point $p_0 \in  M$ that lies at the center of an $(\varepsilon_1, L)$-neck at time $t_0$, but not at the center of an $(\frac{\varepsilon_1}{2}, 2L)$-neck at time $t_0$. 
\end{lemma}
\begin{proof}
Fix a time $t_0$ and for simplicity, let us suppress $t_0$ in our notation. As in Lemma \ref{approx_axis}, let $\omega$ be a unit vector in $\R^{n+1}$ such that $\langle \nu, \omega \rangle \geq 0$ everywhere on $M$. Assume $\varepsilon_1$ is sufficiently small so that $\langle \nu, \omega \rangle \leq \frac{1}{100}$ on any $(2\varepsilon_1, \frac{L}{2})$-neck. First, we claim there exists a point $q_0 \in M$ that does not lie on an $(2\varepsilon_1, \frac{L}{2})$-neck. This is clear for topological reasons. If every point in $M$ lies at the center of an $(2\varepsilon_1, \frac{L}{2})$-neck, then $0 \leq \langle \nu, \omega \rangle \leq \frac{1}{100}$ everywhere on $M$. Let $p \in M$ be an arbitrary point lying at the center of an $(2\varepsilon_1, \frac{L}{2})$-neck $N$. Let $y$ denote the height coordinate defined by $\omega$ normalized so that $y(p) = 0$. The estimate $0 \leq \langle \nu, \omega \rangle \leq \frac{1}{100}$ implies integral curves of $\frac{\omega^\top}{|\omega^\top|^2}$ can be continued indefinitely in either direction of the neck $N$. On the other hand, one connected component of $M \setminus N$ is bounded, and therefore the height $y$ has a one-sided bound. At a critical point of $y$ on the bounded component we find a contradiction.  

Now if $q_0$ does not lie at the center of an $(2\varepsilon_1, \frac{L}{2})$-neck, there is an open neighborhood around $q_0$ that does not contain any points at the center of an $(\varepsilon_1, L)$-neck. On the other hand, Corollary \ref{nearby_neck} and the Neck Detection Lemma imply there exist points in $M$ that lie at the center of $(\varepsilon_1, L)$-necks. Among all such points, let $p_0$ be a point of least intrinsic distance to $q_0$. Then $p_0$ lies at the center of an $(\varepsilon_1, L)$-neck, but does not lie at the center of an $(\frac{\varepsilon_1}{2}, 2L)$-neck. If $p_0$ were at the center of the finer neck, we could express the region around $p_0$ as a graph over an cylinder $\Sigma := S^{n-1} \times [-2L, 2L]$. Following a minimal geodesic connecting $p_0$ to $q_0$ for a little ways and restricting our attention to a suitable subcylinder of $\Sigma$, we will find points that lie at the center of $(\varepsilon_1, L)$-necks closer to $q_0$ than $p_0$, a contradiction. 
\end{proof}


\section{Neck and Cap Decomposition: Compact Case}

In this section, we will assume $M$ is compact and prove a structure theorem analogous to Theorem \ref{neck_cap_noncompact}. Our first lemma shows that if $\mathcal M$ is not a family of shrinking round spheres, then for sufficiently negative times, the solution must contain a very fine neck.  

\begin{lemma}\label{asymptotics}
Suppose $\mathcal M$ is compact and satisfies $(\ast)$. Then either $\mathcal M$ is a family of shrinking round spheres or for every $\eta > 0$, we can find a time $T_\eta \in (-\infty, 0]$ such that for every $t \leq T_\eta$, there exists a point $p \in M$ where $\lambda_1(p, t) < \eta H(p, t)$. 
\end{lemma}
\begin{proof}
By the classical result of Huisken \cite{Hui84}, the inequality $h_{ij} \geq \eta H g_{ij}$ is preserved by the mean curvature flow for compact initial data in $\R^{n+1}$. Suppose for some $\eta > 0$, there exists a sequence of times $t_j \to -\infty$ such that at each time $t_j$ we have $h_{ij} \geq \eta Hg_{ij}$ on $M$. Since this inequality is preserved forward in time, we conclude $h_{ij} \geq \eta H g_{ij}$ for all $t \in (-\infty, 0]$. The main result of \cite{HS15} then implies $\mathcal M$ is a family of shrinking round spheres. 
\end{proof}

In the compact setting, evidently there is no unit vector $\omega$ that satisfies $\langle \nu, \omega \rangle \geq 0$ everywhere to serve as an approximate axis for every neck. Nevertheless, convexity still implies that the axes of different necks cannot differ by much. 

\begin{lemma}\label{approx_axis_compact}
Suppose $F : M \to \R^{n+1}$ is an embedding of a closed, convex hypersurface. Let $L \geq 100$. Suppose $N \subset M$ is an $(\varepsilon, L)$-neck and let $\omega$ denote its axis. There are constants $C := C(n)$ and $\varepsilon(n) > 0$, such that if $0 < \varepsilon < \varepsilon(n)$ and $\tilde N \subset M$ is any other $(\varepsilon, L)$-neck, then $|\langle \omega, \nu \rangle |\leq C \varepsilon$ on $\tilde N$.
\end{lemma}
\begin{proof}
Let $L \geq 100$ be given and let $N \subset M$ be an $(\varepsilon, L)$-neck. The axis of $N$, denoted $\omega$, is a unit vector parallel to the axis of the cylinder over which $N$ is a graph. Let $p$ a point on the central sphere of $N$. It follows from the definition of a neck that:
\begin{enumerate}
\item[(i)] $|\langle \nu(q), \omega \rangle |\leq  C\varepsilon$ for every $q \in N$.
\item[(ii)] For every unit vector $e$ orthogonal to $\omega$, there exists a point $q \in N$ with $d_g(p, q) \leq 2\pi r_p$ where $|\nu(q) - e| \leq C\varepsilon$ (recall $r_p = \frac{n-1}{H(p)}$).  
\end{enumerate} 
Now suppose $\tilde N$ is any other $(\varepsilon, L)$-neck in $M$. Let $\tilde p$ be a point on its central sphere and $\tilde \omega$ its axis. We may assume $\langle \omega, \tilde \omega \rangle \geq 0$. Let $y(q) = \langle F(q) - F(p), \omega \rangle$ denote the height function with respect to $\omega$ and let $\Sigma_{0}$ denote the level set $y = 0$. It suffices to prove that $|\omega - \tilde \omega| \leq C\varepsilon$. Evidently, we may assume $\omega \neq \tilde \omega$, otherwise there is nothing to prove. 

First, suppose $d_g(\tilde p, \Sigma_0) \leq \frac{L}{2}r_p$. This means $\tilde p \in N$ and so, by taking $\varepsilon$ small, $H(\tilde p)$ is as close as we like to $H(p)$. In this case, $B_g(\tilde p, 3 \pi r_{\tilde p}) \subset B_g(\tilde p, 4 \pi r_p) \subset N$ since $\tilde p$ has distance at least $\frac{L}{4}r_p \geq 25r_p$ from boundary of $N$. Now if $e$ is a unit vector orthogonal to $\tilde \omega$, it follows from (ii) that we can find a point $q \in B_g(\tilde p, 3\pi r_{\tilde p})$ where $|\nu(q) - e| \leq C\varepsilon$. Since $q \in N$, by property (i), we conclude $|\langle e, \omega \rangle| \leq C \varepsilon$. Summing over all directions orthogonal to $\tilde \omega$ gives $|\omega - \tilde \omega| \leq C \varepsilon$. 

Next, suppose $d_g( \tilde p, \Sigma_0) > \frac{L}{2} r_p$. We may assume $\omega$ is orientated so that $y(\tilde p) > 0$. Let $v := \omega  - \langle \omega, \tilde \omega \rangle \tilde \omega$. By (ii), we can find a point $q \in \tilde N$ with $d_g(\tilde p, q) \leq 2r_{\tilde p}$ where $|\nu(q) + \frac{v}{|v|}| \leq C \varepsilon$. We cannot have $y(q) < 0$. Otherwise, by convexity of $M$ a minimal geodesic connecting $\tilde p$ to $q$ (which must remain in $\tilde N$) crosses the neck $N$ and so $d_g(\tilde p, q) \geq \frac{L}{2} r_p$. Hence $\frac{L}{2} r_p \leq 2 r_{\tilde p}$, or, equivalently, $\frac{H(\tilde p)}{H(p)} \leq \frac{4\pi}{L}\leq \frac{1}{5}$. However, since $N$ and $\tilde N$ intersect (and $\varepsilon$ is small) $\frac{H(\tilde p)}{H(p)} \geq \frac{1}{2}$, which gives a contradiction. Thus, $y(q) > 0$. In the compact and convex setting, the integral curves of $\frac{\omega^\top}{|\omega^\top|^2}$ emanating from $\Sigma_0$ cover $M \cap \{y \geq 0\}$. Using that $\frac{d}{dy} \langle \nu, \omega \rangle \geq \lambda_1 > 0$, we get $\langle \nu(q), \omega\rangle \geq - C\varepsilon$. Finally, as before
\[
- C \varepsilon \leq \langle \nu(q), \omega \rangle = - \big\langle \frac{v}{|v|}, \omega \big\rangle + \big\langle \nu(q) + \frac{v}{|v|}, \omega \big\rangle \leq -\sqrt{1 - \langle \omega, \tilde \omega \rangle^2} +  C \varepsilon.
\]
Hence $|\omega - \tilde \omega | \leq C \varepsilon$. This completes the proof. 
\end{proof}

We now prove the structure theorem for compact $\mathcal M$ satisfying $(\ast)$. 

\begin{theorem}\label{neck_cap_compact}
Suppose $\mathcal M$ is compact, satisfies $(\ast)$, and is not a family of shrinking round spheres. Given $0 < \varepsilon_0 < \varepsilon(n)$ small and $L \geq 100$, there exist constants $C_0 < \infty$ and $T_0 \leq 0$, depending only upon $\varepsilon_0$, $L$, $n$, $\gamma_1$ and $\gamma_2$, so that the following holds. For any time $t \in (-\infty, T_0]$, we can find disjoint domains $D_{1}, D_{2} \subset M$, and points $p_1, p_2 \in M$ that lie at the center of necks $N_1, N_2$ such that
\begin{enumerate}
\item[(1)] Every point in $M \setminus (D_{1} \cup D_{2})$ lies at the center of an $(\varepsilon_0, L)$-neck at time $t$. 
\item[(2)] $D_1$ and $D_2$ are diffeomorphic to $B^n$. 
\item[(3)] $\d D_1 \subset \d N_1$ and $\d D_2 \subset \d N_2$ are cross-sectional spheres of $(\varepsilon_0, L)$-necks. 
\item[(4)] For $i = 1, 2$, the intrinsic diameter of $D_i$ is bounded by $C_0 H(p_i, t)^{-1}$. 
\item[(5)] For $i = 1, 2$, every point $q \in D_i$ satisfies $C_0^{-1} H(p_i,t) \leq H(q,t) \leq C_0 H(p_i,t)$ and $\lambda_1(q,t) \geq C_0^{-1} H(q,t)$. 
\end{enumerate}
\end{theorem}
\begin{proof}
Let $0 < \varepsilon_0 < \varepsilon(n)$ small and $L \geq 100$ be given. As in the proof of Theorem \ref{neck_cap_noncompact}, we begin by fixing certain parameters using the Neck Detection Lemma and the definition of a neck. So that the dependence of the constants is clear, we enumerate our choices. 
\begin{enumerate}
\item[$\bullet$] Using the definition of a neck and Lemma \ref{approx_axis_compact}, we assume $\varepsilon_0 < \varepsilon(n)$ is sufficiently small so that if $\omega$ is the axis of an $(\varepsilon_0, L)$-neck, then $|\langle \nu, \omega \rangle | \leq \frac{1}{100}$ everywhere on the neck. Moreover, we require that if $\tilde \omega$ is the axis any other $(\varepsilon_0, L)$-neck, then up to a choice of orientation, we have $|\omega - \tilde \omega| \leq \frac{1}{100}$. We also assume parts (i) and (iii) of Lemma \ref{curves} hold with respect to the axis of any $(\varepsilon_0, L)$-neck (part (ii) holds automatically) and $\frac{9}{10} \leq \frac{H(q_1)}{H(q_2)} \leq \frac{10}{9}$ for any pair of points $q_1, q_2$ on an $(\varepsilon_0, L)$-neck. 
\item[$\bullet$] Using the Neck Detection Lemma, we choose $\eta_0:= \eta_0(\varepsilon_0, n, L, \gamma_1, \gamma_2) \in(0, \varepsilon_0)$ so that if $\lambda_1(q) \leq \eta_0 H(q)$, then $q$ lies at the center of an $(\varepsilon_0, L)$-neck. 
\item[$\bullet$] Using both the definition of a neck and the Neck Detection Lemma, we choose $\varepsilon_1 := \varepsilon_1(n, \eta_0) \in (0, \eta_0)$ and $\eta_1:= \eta_1(\varepsilon_1, n, L, \gamma_1, \gamma_2) \in (0, \varepsilon_1)$ sufficiently small so that on any $(\varepsilon_1, L)$-neck, we have $\lambda_1 \leq \eta_0 H$ and if $\lambda_1(q) \leq \eta_1 H(q)$, then $q$ lies at the center of an $(\varepsilon_1, L)$-neck. 
\item[$\bullet$] Using both the definition of a neck and the Neck Detection Lemma, we choose $\hat \varepsilon_1 := \hat \varepsilon_1(n, \eta_1) \in (0, \eta_1)$ and $\hat \eta_1:= \hat \eta_1(\hat \varepsilon_1, n, L, \gamma_1, \gamma_2) \in (0, \hat \varepsilon_1)$ sufficiently small so that on any $(\hat \varepsilon_1, L)$-neck, we have $\lambda_1 \leq \frac{1}{2}\eta_1 H$ and if $\lambda_1(q) \leq \hat \eta_1 H(q)$, then $q$ lies at the center of an $(\hat \varepsilon_1, L)$-neck. 
\item[$\bullet$] Using the definition of a neck, we choose $\varepsilon_2 := \varepsilon_2(\hat \eta_1, \eta_1, n)$ so that the axis $\omega$ of an $(\varepsilon_2, L)$-neck satisfies $|\langle \nu, \omega \rangle|  \leq \hat \eta_1$ on $N$ and so that $\lambda_1 \leq \frac{1}{2}\eta_1 H$ everywhere on the neck.
\item[$\bullet$] Using both the definition of a neck and the Neck Detection Lemma, we choose $\varepsilon_3 := \varepsilon_3(n, \hat \eta_1) \in (0, \hat \eta_1)$ and $\eta_3 := \eta_3(\varepsilon_3, n, L, \gamma_1, \gamma_2)\in (0, \varepsilon_3)$ so that if $\lambda_1(q) \leq \eta_3H(q)$, then $q$ lies at the center of an $(\varepsilon_3, L)$-neck $N$ for which the axis $\omega$ satisfies $|\langle \nu, \omega \rangle| \leq \hat \eta_1$. 
\item[$\bullet$] Finally, using Lemma \ref{asymptotics} and the Neck Detection Lemma, we choose $T_0:= T_0(\varepsilon_2) \leq 0$, such that if $t \leq T_0$, then we can find a point that lies at the center of an $(\varepsilon_2, L)$-neck. 
\end{enumerate} 

Fix a sufficiently negative time $t \leq T_0$. We can find a point $p \in M$ that lies at the center of an $(\varepsilon_2, L)$-neck $N$ at time $t$. Having fixed it, let us suppress $t$ from our notation. Let $\omega$ be the axis of $N$ satisfying $|\langle \nu, \omega \rangle| \leq \hat \eta_1$. Our goal is to follow the neck $N$ in the direction $\omega$ until we encounter a point which barely lies on an $(\varepsilon_1, L)$-neck. As in the notation of the previous section, let $y$ denote the height function with respect to axis defined by $\omega$; let $\Sigma_y$ for $y \in [0, y_{\max})$ denote the level sets of the height function (with $p \in \Sigma_0$); and let $\gamma(\cdot, y)$ denote the integral curves of $\frac{\omega^\top}{|\omega^\top|^2}$. Note $y_{\max} < \infty$ because $M$ is compact. By definition every integral curve is defined for $y \in [0, y_{\max})$ and there exists some point $q \in \Sigma_0$ such that $\gamma(q, y) \to \tilde q \in M$ as $y \to y_{\max}$ satisfying $\langle \nu(\tilde q), \omega \rangle = 1$.  As we noted in the proof of Theorem \ref{neck_cap_noncompact}, since our hypersurface is strictly convex, elementary arguments imply $\Sigma_{y_{\max}} := \{ \lim_{y \to y_{\max}} \gamma(q, y) : q \in \Sigma_0\} = \{\tilde q\}$; so every curve converges to the same point. 

Let $\tilde y \in [0, y_{\max})$ be the supremum over heights $y$ such that $\lambda_1 \leq \eta_1 H$ holds for every $y \in [0, \tilde y]$. Since $\lambda_1 \leq \frac{1}{2} \eta_1 H$ on $N$, we have $\tilde y \geq \frac{1}{2}L r_p > 0$. On the other hand, if $\tilde q$ lies on an $(\varepsilon_1, L)$-neck, then the axis of such a neck would be nearly orthogonal to $\omega$. However, this  contradicts our application of Lemma \ref{approx_axis_compact} in the first bullet point above. Hence $\lambda_1(\tilde q) > \eta_0 H(\tilde q)$ and $\tilde y < y_{\max}$. 

On $\Sigma_{\tilde y}$, we find a point $\tilde p$ satisfying $\lambda_1(\tilde p) = \eta_1 H(\tilde p)$, which barely lies at the center of $(\varepsilon_1, L)$-neck $\tilde N$. As we have argued before, this implies $\lambda_1 \geq \hat \eta_1 H \geq \frac{1}{2} \hat \eta_1 H(p)$ on $\Sigma(\tilde y - 2 r_{\tilde p}, \tilde y)$. Otherwise, there is a point on $\tilde N$ near $\tilde p$ which lies on an $(\hat \varepsilon_1, L)$-neck, which would contradict $\lambda_1(\tilde p) = \eta_1 H(\tilde p)$. Now as in Step 2 of Theorem \ref{neck_cap_noncompact}, integrating this convexity lower bound along the integral curves of $\omega$, we get $\langle \nu, \omega \rangle > (n-2)\hat \eta_1$ on $\Sigma_{\tilde y}$. We can obtain the uniform convexity, mean curvature, and diameter estimates for the region $\Sigma(\tilde y, y_{\max})$ in the same fashion as Step 1 in the proof of Theorem \ref{neck_cap_noncompact}. The estimate $\langle \nu, \omega \rangle > (n-2)\hat \eta_1$ holds for $y \in [\tilde y, y_{\max}]$ because $\langle \nu, \omega \rangle$ is increasing along integral curves. The mean curvature lower bound, $H \geq \theta^{-1} H(\tilde p)$, follows from part (iii) of Lemma \ref{curves}. If a point $q$ in the region $\Sigma(\tilde y, y_{\max})$ satisfies $\lambda_1(q)\leq \eta_3 H(q)$, then $q$ lies at the center of an $(\varepsilon_3, L)$-neck $\hat N$. Our assumptions imply axis $\hat \omega$ of this neck cannot equal $\omega$ otherwise we contradict the estimate $\langle \nu, \omega \rangle > (n-2)\hat \eta_1$. However, if $\omega$ and $\hat \omega$ differ, then by arguments we have made before, $(n-2) \hat \eta_1 \leq - \sqrt{1 - \langle \omega, \hat \omega \rangle^2} + \hat \eta_1$, which is a contradiction. Therefore $\lambda_1 > \eta_3 H$ for $y \in [\tilde y, y_{\max}]$. In Corollary \ref{nearby_neck}, choose $\hat a$ and $\hat b$ with respect to the constants $\frac{1}{2}\eta_3$ and $\gamma_1$. Let $q \in \Sigma(\tilde y, y_{\max})$. Then either $M = B_g(q, \hat a H(q)^{-1})$ or there exists a point $q_1 \in B_g(q, \hat a H(q)^{-1})$ where $\lambda_1(q_1) \leq \frac{1}{2} \eta_3H(q_1)$. In either case, $B_g(q, \hat a H(q)^{-1})$ has nonempty intersection with the neck $\tilde N$. Arguing as in the last part of Step 1 in the proof of Theorem \ref{neck_cap_noncompact} yields the desired upper mean curvature bound and the diameter estimate for the region $\Sigma(\tilde y, y_{\max})$.

To finish, we let $p_1 := \tilde p$ and let $D_1 \subset M \setminus \tilde N$ be the connected component containing $\tilde q$. Our arguments show claims (2)-(5) in the theorem statement hold for $D_1$ and $p_1$. Repeating the entire argument above for $y \leq 0$, we similarly find $p_2$ and $D_2$. Moreover, since $p_1$ and $p_2$ are the first points that barely lie at the center of $(\varepsilon_1, L)$-necks, our argument ensures that every point in $M \setminus (D_1 \cup D_2)$ lies at the center of an $(\varepsilon_0, L)$-neck, which is (1). This completes the proof of the theorem. 
\end{proof}

Theorem \ref{neck_cap_compact} address the structure of the ancient solution at sufficiently negative times. The following corollary addresses the structure of the solution at later times. 

\begin{corollary}
Suppose $\mathcal M$ is compact and satisfies $(\ast)$. Given $0 < \varepsilon_0 < \varepsilon(n)$ and $L \geq 100$, there exists a constant $C < \infty$, depending only upon $\varepsilon_0$, $L$, $n$, $\gamma_1$, and $\gamma_2$ so that the following holds. For every time $t \in (-\infty, 0]$, either there exists domains $D_1, D_2$, points $p_1, p_2$, and necks $N_1, N_2$ satisfying the conclusions of Theorem \ref{neck_cap_compact}, or else, for any point $p \in M$, there holds:
\begin{enumerate}
\item[(1)] The intrinsic diameter of $M$ is bounded by $C_0 H(p, t)^{-1}$. 
\item[(2)] Every point $q \in M$ satisfies $C_0^{-1} H(p, t) \leq H(q, t) \leq C_0 H(p, t)$ and $\lambda_1(q, t) \geq C_0^{-1} H(q, t)$. 
\end{enumerate}
\end{corollary}
\begin{proof}
Let $0 < \varepsilon_0 < \varepsilon(n)$ and $L \geq 100$ be given (where $\varepsilon(n)$ is determined in the proof of Theorem \ref{neck_cap_compact}). We follow the first part of the proof of Theorem \ref{neck_cap_noncompact} to determine a constant $\varepsilon_2$ depending only upon the given constants. Since $t \leq T_0$ is addressed by the previous theorem, suppose $T_0 < t \leq 0$. If there exists a point $p \in M$ that lies at the center of an $(\varepsilon_2, L)$-neck at time $t$, then the proof of the previous theorem goes through unchanged, and we conclude there exists $D_i, p_i$, and $N_i$ for $i = 1, 2$ at time $t$. If, however, there does not exist a point in $M$ that lies at the center of an $(\varepsilon_2, L)$-neck, then by the Neck Detection Lemma, we can find a constant $\eta := \eta(\varepsilon_0, L, n, \gamma_1, \gamma_2) > 0$, such that $\lambda_1 \geq \eta H$ everywhere on $M$. Let $p$ be an arbitrary point in $M$. By Corollary \ref{nearby_neck}, we can find constants $\hat a$ and $\hat b$ depending upon $n, \eta$, and $\gamma_1$ such that $M = B_{g(t)}(p, \hat a H(p, t)^{-1})$ and $H(q, t) \geq \hat b^{-1} H(p, t)$ for every $q \in M$. We can take $C_0 := \max \{\eta^{-1}, \hat a, \hat b \}$. This completes the proof. 
\end{proof}


\section{Noncollapsing}

To complete the proof of our main theorem, we first show how controlled geometry of the cap implies $\alpha$-noncollapsing for an appropriate $\alpha$. 
\begin{theorem}\label{noncollapsed_hypersurface}
Let $n \geq 2$ and $F : M \to \R^{n+1}$ be an embedding of a closed, convex hypersurface. Given a positive constant $C_0$ such that $C_0^{-1} \leq H \leq C_0$ and $\lambda_1 \geq C_0^{-1} H$ everywhere on $M$, there exists $\alpha := \alpha(C_0)>0$ such that $M$ is $\alpha$-noncollapsed. 
\end{theorem}
\begin{proof}
Let $\Omega$ denote the convex interior of $F(M)$ and let $\nu$ denote the outward pointing normal vector. Let $\alpha > 0$. Throughout the proof, we will assume $\alpha$ is sufficiently small depending upon $C_0$. For any point $p \in M$, let $B_{p,\alpha}$ denote the Euclidean $(n+1)$-ball $B(F(p) - \alpha C_0 \nu(p), \alpha C_0) \subset \R^{n+1}$. It suffices to find $\alpha$ such that $B_{p,\alpha} \subset \Omega$ for every $p \in M$. By definition, $M$ is $\alpha$-noncollapsed if, for every $p \in M$, the Euclidean $(n+1)$-ball $B(F(p) - \frac{\alpha}{H(p)} \nu(p), \frac{\alpha}{H(p)})$ is contained in $\Omega$. Because, by assumption, $\frac{\alpha}{H(p)} \leq \alpha C_0$, it is easy to see $B(F(p) - \frac{\alpha}{H(p)} \nu(p), \frac{\alpha}{H(p)}) \subset B_{p, \alpha}$. 
Suppose there exist points $p \in M$ such that $B_{p, \alpha}$ is not contained in $\Omega$. We will examine short unit-speed geodesics $\gamma(s)$ emanating from $p$ for $s \in [0, \alpha C_0]$. We distinguish two cases. In the first case, we will suppose that there exists a short geodesic emanating from some point $p$ which bends into the region $B_{p, \alpha}$. We will show that for $\alpha$ sufficiently small, this would imply the curvature nearby is too large. In the second case, we will assume for every $p$ that short geodesics do not intersect $B_{p, \alpha}$. We will use the lower bound for $\lambda_1$ to deduce a height estimate for short geodesics in $M$ over $T_pM$ for each $p \in M$. The height estimate, convexity, and non-intersection with $B_{p,\alpha}$ will imply the surface is $\tilde \alpha$-noncollapsed for a suitable $\tilde \alpha > 0$. 

\textit{Case 1:} Suppose there exists a point $p \in M$ with $B_{p, \alpha} \not \subset \Omega$ and a unit vector $e \in T_pM$ such that the unit-speed geodesic $\gamma(s)$, defined by $\gamma(0) = p$ and $\gamma'(0) = e$, enters the region enclosed by the ball $B_{p,\alpha}$ in the interval $[0, \alpha C_0]$. That is, there exists some $s \in (0, \alpha C_0]$ such that $|F(\gamma(s)) - F(p) + \alpha C_0 \nu(p)| \leq \alpha C_0$. Let $\bar \nabla$ denote the ambient connection on $\R^{n+1}$. Since $F$ is an embedding, we write $\gamma'(s)$ in place of $dF(\gamma'(s))$. Using that $\nabla_{\gamma'} \gamma' = 0$, we have
\[
\bar \nabla_{\gamma'}\gamma' = \nabla_{\gamma'} \gamma' + \langle \bar \nabla_{\gamma'} \gamma', \nu \rangle \nu = - h(\gamma', \gamma') \nu.
\]
Consider the function 
\[
f(s) = \frac{1}{2}\big(|F(\gamma(s))- F(p) + \alpha C_0 \nu(p)|^2 - \alpha^2C_0^2\big).
\]
The derivatives of this function are given by 
\begin{align*}
f'(s) &= \big\langle \gamma'(s)\,,\, F(\gamma(s)) -F(p) + \alpha C_0\nu(p) \big\rangle, \\
f''(s) &= \big \langle \bar \nabla_{\gamma'}  \gamma'(s)\,,\, F(\gamma(s)) - F(p) + \alpha C_0 \nu(p) \big \rangle + | \gamma'(s)|^2 \\
& = -h(\gamma'(s), \gamma'(s))\big \langle \nu(\gamma(s))\,,\,F( \gamma(s)) -F(p) + \alpha C_0 \nu(p) \big\rangle + 1.
\end{align*}
Initially we have $f(0) = 0$, $f'(0) = 0$, and $f''(0) =1 - \alpha C_0 h(e,e)$. Since $h(e,e) < H(p) \leq C_0$, $f''(0) > 1 - \alpha C_0^{2}$. We assume $\alpha \leq \frac{1}{2}C_0^{-2}$ so that $f''(0) \geq \frac{1}{2}$. This implies $f(s) > 0$ for small $s > 0$. By the mean value theorem, the function attains a positive local maximum at some $s_\ast \in (0, \alpha C_0)$.  At a local maximum, $f''(s_\ast) \leq 0$, which implies 
\[
\langle \nu(\gamma(s_\ast))\,,\, F(\gamma(s_\ast)) - F(p) + \alpha C_0 \nu(p) \rangle \,h(\gamma'(s_\ast), \gamma'(s_\ast)) \geq 1 
\]
Note $h(\gamma'(s_\ast), \gamma'(s_\ast)) >0$, so $\langle \nu(\gamma(s_\ast))\,,\, F(\gamma(s_\ast)) - F(p) + \alpha C_0 \nu(p) \rangle > 0$. Now because $\gamma$ is a unit-speed curve, $|F(\gamma(s_\ast)) - F(\gamma(0))| \leq s_\ast \leq \alpha C_0$. This gives
\[
\langle \nu(\gamma(s_\ast))\,,\, F(\gamma(s_\ast)) - F(p) + \alpha C_0 \nu(p) \rangle \leq |F(\gamma(s_\ast)) - F(p) + \alpha C_0 \nu(p)| \leq  2 \alpha C_0 \leq C_0^{-1}.
\]
Therefore, $h(\gamma'(s_\ast), \gamma'(s_\ast)) \geq C_0$. This contradicts our previous observation $h(\gamma'(s_\ast), \gamma'(s_\ast))  < H(\gamma(s_\ast)) \leq C_0$. In conclusion, this case cannot occur if $\alpha \leq \frac{1}{2}C_0^{-2}$. 

\textit{Case 2:} Let $\alpha = \frac{1}{2}C_0^{-2}$. The alternative is that around every point $p$, the unit-speed geodesics emanating from $p$ do not enter $B_{p, \alpha}$ for $s \in [0, \alpha C_0]$. In this case, we examine the height of $M$ over $T_pM$ for every $p \in M$. Fix a point $p$, a unit vector $e \in T_pM$, and consider the unit-speed geodesic $\gamma(s)$ emanating from $p$ for $s \in [0, \alpha C_0]$ as above. Define 
\[
k(s) = \langle F(\gamma(s)) - F(p), -\nu(p) \rangle. 
\]
The function is the height of the point $F(\gamma(s))$ over the hyperplane $dF(T_pM) \subset \R^{n+1}$. The derivatives of this function are given by 
\begin{align*}
k'(s) &= \langle \gamma'(s), - \nu(p) \rangle, \\
k''(s) &= \big\langle \bar \nabla_{\gamma'(s)} \gamma'(s), - \nu(p) \big\rangle \\
& = h(\gamma'(s), \gamma'(s)) \big\langle \nu(\gamma(s)), \nu(p) \big\rangle.
\end{align*}
Initially, $k(0) = 0$, $k'(0) = 0$, and $k''(0) = h(e, e) \geq \lambda_1(p) \geq C_0^{-1} H \geq C_0^{-2}$. Let $e_1, \dots, e_n$ be an orthonormal frame of $T_{\gamma(s)}M$ for some $s \in [0, \alpha C_0]$. Then
\begin{align*}
\Big|\frac{d}{ds} \nu(\gamma(s)) \Big|&= \Big|\bar \nabla_{\gamma'(s)} \nu(\gamma(s))\Big| = \Big|\sum_{i =1}^n h(\gamma'(s), e_i) e_i\Big| \leq \lambda_n(\gamma(s)) \leq C_0.
\end{align*}
This implies 
\[
\langle \nu(\gamma(s)), \nu(p) \rangle \geq 1 - C_0 s \geq 1 - \alpha C_0^2 \geq \frac{1}{2}.
\]
It follows that for $s \in [0, \alpha C_0]$
\[
k''(s) \geq \frac{1}{2} h(\gamma'(s), \gamma'(s)) \geq \frac{1}{2} \lambda_1(\gamma(s)) \geq \frac{1}{2}C_0^{-1}H \geq \frac{1}{2} C_0^{-2}.
\]
Integrating implies  
\[
k(s) = \int_0^s k'(u) \, du = \int_0^s \int_0^u k''(v) \, dv \, du \geq \frac{1}{4}s^2 C_0^{-2}.
\]
Consequently, 
\[
\langle F(q) - F(p), - \nu(p) \rangle \geq \frac{1}{4}C_0^{-2} d_g(p,q)^2 
\]
for points $q \in M$ satisfying $d_g(p, q) \leq \alpha C_0 = \frac{1}{2} C_0^{-1}$. In other words, the height of the boundary of $F(B_g(p, \frac{1}{2}C_0^{-1}))$ over $dF(T_pM)$ is at least $\frac{1}{16}C_0^{-4}$. Since we have assumed that the short unit-speed geodesics do not enter the ball $B_{p, \alpha}$, convexity of $\Omega$ implies that the portion of $B_{p, \alpha}$ within the halfspace $\{x \in \R^{n+1} : \langle x - F(p), - \nu(p) \rangle \leq \frac{1}{16}C_0^{-4}\}$ is contained in $\Omega$. If $\tilde \alpha := \frac{1}{32}C_0^{-5}$, then the ball $B_{p, \tilde \alpha}$ is contained within this region. The point $p \in M$ was arbitrary and thus $M$ is $\tilde \alpha$-noncollapsed. This completes the proof. 
\end{proof}

The proof of the above result is local in the sense that it only depended upon examining the hypersurface $M$ in a small geodesic ball around each point $p$. From the proof, one readily deduces the following corollary. 

\begin{corollary}\label{local_noncollapsing}
Let $n \geq 2$ and $F : M \to \R^{n+1}$ be an embedding of a (possibly noncompact) complete, convex hypersurface. Let $D \subset M$ be an open region and suppose there exists a constant $C_0$ such that $C_0^{-1} \leq H \leq C_0$ and $\lambda_1 \geq C_0^{-1} H$ everywhere on $D$. Then there exists $\alpha := \alpha(C_0) > 0$ such that if $p \in D$ is any point with $d_g(p, \d D) > \frac{1}{2} C_0^{-1}$, then the inscribe radius at $p$ is at least $\frac{\alpha}{H(p)}$. 
\end{corollary}

The corollary shows that points on a cap have a uniform inscribe radius. On the other hand, a point that lies at the center of a fine neck has an inscribe radius that is comparable to the inscribe radius of the cylinder. One could argue this in the following way: if $p \in M$ lies at the center of an $(\varepsilon, L)$-neck $N$, then we can find an approximate axis $\omega$ of $N$ such that $\langle \nu(p), \omega \rangle = 0$. Let $y(x) = \langle x - F(p), \omega \rangle$ denote the Euclidean coordinate defined by $\omega$; let $P_{\bar y} := \{x \in \R^{n+1} : y(x) = \bar y\} \subset \R^{n+1}$ denote hyperplanes orthogonal to $\omega$; and let $\Sigma_{\bar y} := F(M) \cap P_{\bar y}$. For small $\varepsilon$, $L \geq 100$, and $|y| \leq 5 \frac{n-1}{H(p)} < L \frac{n-1}{H(p)}$, the surface $\Sigma_y$ is very close to an $(n-1)$-sphere of radius $\frac{n-1}{H(p)}$. Consequently, $\Sigma_0$ is embedded within $P_0$, bounds a domain $\Omega_0 \subset P_0$, and $\Omega_0$ must contain a Euclidean $n$-ball of radius $\frac{n-1}{2H(p)}$ tangent to $F(p)$ within $P_0$. Now follow the integral curve $\gamma(y, p)$ of $\frac{\omega^\top}{|\omega^\top|^2}$ for $|y| \leq \frac{n-1}{H(p)}$. Along $\gamma$, the mean curvature changes very little. So for $\varepsilon$ sufficiently small, the domain $\Omega_y$ in the hyperplane $P_y$ contains a Euclidean $n$-ball tangent to $F(\gamma(y, p)) \in \Sigma_y$ of radius $\frac{n-1}{4H(p)}$. Taking the union over the $\Omega_y$ for $|y| \leq \frac{n-1}{H(p)}$ , we construct a small tube contained $\Omega$, the convex interior of $F(M)$. The tube curves a little bit in the direction of $\omega$, but because we are on a fine neck, this curvature is very small compared to the curvature of the Euclidean $(n+1)$-ball $B:= B(F(p) - \frac{n-1}{8H(p)}\nu(p), \frac{n-1}{8H(p)})$. It is easy to see that $P_y \cap B \subset \Omega_y$ and therefore $B$ is contained in the intersection of $\Omega$. By these considerations, the inscribe radius at $p$ is at least $\frac{n-1}{8H(p)}$. 

With our structure theorem and these noncollapsing arguments, we can show the time slices of a solution $\mathcal M$ are all noncollapsed with a uniform noncollapsing constant. This will complete the proof of Theorem \ref{main}. 

\begin{theorem}
Suppose $\mathcal M$ satisfies $(\ast)$. There exists $\alpha > 0$, depending only upon $n$, $\gamma_1$, and $\gamma_2$, such that $\mathcal M$ is $\alpha$-noncollapsed. 
\end{theorem}
\begin{proof}
First suppose $\mathcal M$ is noncompact. Fix a time $t_0$ and let $\Omega$ denote the convex interior of $F(M, t_0)$. Set $L = 100$ and fix some $0 < \varepsilon_0 < \varepsilon(n)$ sufficiently small for Theorem \ref{neck_cap_noncompact} to apply. For these values of $L$ and $\varepsilon_0$, we can find constants $\varepsilon_1 \in (0, \varepsilon_0)$ and $C_0 > 1$, depending only upon $n, \gamma_1$, and $\gamma_2$, so that the conclusions of Theorem \ref{neck_cap_noncompact} hold.  We can assume $\varepsilon_1$ is small enough for Lemma \ref{existence_neck} to apply. Then there exists a point $p \in M$ that lies at the center of an $(\varepsilon_1, L)$-neck, but not at the center of an $(\frac{\varepsilon_1}{2}, 2L)$-neck (at time $t_0$). 

After a rescaling, we may assume $H(p) = 1$. The results of Theorem \ref{neck_cap_noncompact} tell us $M$ is the union of a compact connected component $D$, the neck $N$, and an unbounded connected component $\tilde D$. In $\tilde D \cup N$, every point lies at the center of an $(\varepsilon_0, L)$-neck. Moreover, the estimates $C_0^{-1} \leq H \leq C_0$ and $\lambda_1 \geq C_0^{-1}H$ hold in $D \cup N$.

First, suppose $q \in M$ lies at the center of an $(\varepsilon_0, L)$-neck $N$. An exact cylinder is $\alpha$-noncollapsed for $\alpha = n-1$. As in the remark before the theorem, the inscribe radius of $q$ is at least $\frac{\alpha}{H(q)}$ for some $\alpha = \alpha(n)$. Now suppose $q \in D$ is a point on the cap. The intrinsic length of the neck $N$ is approximately $100(n-1)$, where as $\frac{1}{2} C_0^{-1} < \frac{1}{2}$. Thus clearly $d_g(q, \d \tilde D) > \frac{1}{2} C_0^{-1}$. By Corollary \ref{local_noncollapsing} applied to the region $D \cup N$, the inscribe radius of $q$ is at least $\frac{\alpha(C_0)}{H(q)}$. Since every point in $M$ is either contained in $D$ or at the center of an $(\varepsilon_0, L)$-neck, the hypersurface $M$ is $\alpha$-noncollapsed everywhere for $\alpha$ independent of the time $t_0$. Hence $\mathcal M$ is $\alpha$-noncollapsed. 

If $\mathcal M$ is compact, the argument is similar. We may assume $\mathcal M$ is not a family of shrinking round spheres. For $L = 100$,  $0 < \varepsilon_0 < \varepsilon(n)$ sufficiently small, and for $t_0$ sufficiently negative, we conclude $M$ is the union of a neck region and two caps $D_1$, $D_2$, which have uniform mean curvature and convexity estimates and whose boundaries are spherical cross-sections of $(\frac{\varepsilon_0}{2}, 2L)$-necks. The arguments above show each region is $\alpha$-noncollapsed and hence $\mathcal M$ is $\alpha$-noncollapsed for all sufficiently negative times. Since the noncollapsing constant is preserved forward in time, this completes the proof. 
\end{proof}


\section{appendix}

In this appendix, we explain how to show the sets $E_{p, t}$ introduced in the proof of Lemma \ref{optimal_two_convexity} are invariant under parallel transport (with respect to $g(t)$) using the strict maximum principle. A reference for this argument is \cite{Bre10}. For the convenience of the reader, we make the minor modifications necessary for our setting. 

We begin by introducing relevant notation.  Suppose we have a local solution to the mean curvature flow $\mathcal M$ in $\R^{n+1}$ defined on $\Omega \times (-T, 0)$, for $\Omega$ a smooth domain in $\R^n$ and $T > 0$. Suppose we know our solution is weakly convex and that $\lambda_1 + \lambda_2 \equiv \beta H$ for some $\beta \in (0, \frac{1}{n-1})$. Define a vector bundle $\mathcal E$ over $\Omega \times (-T, 0)$ to be the pullback of the tangent bundle $T\Omega$ under the projection $\Omega \times (-T, 0) \to \Omega$ so that $\mathcal E_{(p, t)} = T_p\Omega$. We have a bundle metric $g(t)$ defined on $\mathcal E$ and there is a standard compatible connection $D$ on $\mathcal E$ that extends the Levi-Civita connection. Namely, for any section $X$ of $\mathcal E$, we define $D_{\pdv{t}}X = \pdv{t} X - \sum_{k =1}^nHh(X, e_i) e_i$ where $\{e_1, \dots, e_n\}$ is an orthonormal frame with respect to $g(t)$. It is easy to check that $D_{\pdv{t}} g(t) = 0$. Finally, let $\mathcal O$ denote the orthonormal frame bundle associated to $\mathcal E$. For every $(p, t)$, the fiber $\mathcal O_{(p,t)}$ consists of orthonormal frames $\underline e = \{e_1, \dots, e_n\}$ of $T_p\Omega$ with respect to $g(t)$. Given a point $(p, t)$ and an orthonormal frame $\underline e = \{e_1, \dots, e_n\}$ with respect to $g(t)$, the tangent space $T_{\underline e} \mathcal O$ decomposes into the direct sum of vertical and horizontal vector spaces, $\mathcal V_{\underline e}$ and $\mathcal H_{\underline e}$. The vertical space is the tangent space to the fiber $\mathcal O_{(p,t)}$ and vertical vectors are induced by infinitesimal action of $O(n)$ upon $\underline e$. The horizontal space is defined through the connection $D$. To define it, consider any smooth path $\gamma : (-\varepsilon, \varepsilon) \to \Omega \times (-T, 0)$ such that $\gamma(0) = (p, t)$. Extend the frame $\underline e$ by parallel transport along $\gamma$ using $D$. This defines a path $\tilde \gamma : (-\varepsilon, \varepsilon) \to \mathcal O$. The vector $\tilde \gamma'(0)$ is defined to be the horizontal lift of $\gamma'(0)$. Define $\mathcal X_1, \dots, \mathcal X_n$ and $\mathcal Y$ in $T_{\underline e} \mathcal O$ to be the horizontal lifts of $e_1, \dots, e_n$ and $\pdv{t}$, respectively. Then one has $T_{\underline e}\mathcal O = \mathcal V_{\underline e} \oplus \mathrm{span}\{\mathcal X_1, \dots, \mathcal X_n, \mathcal Y\}$. 

With the notation above, for each orthonormal frame $\underline e \subset \mathcal E_{(p,t)}$, we define 
\[
\varphi(\underline e) = h_{g(t)}(e_1, e_1) + h_{g(t)}(e_2, e_2) - \beta H(p,t). 
\]
This defines a smooth, nonnegative function $\varphi : \mathcal O \to \R$. Recall the evolution equations $D_{\pdv{t}} h = \Delta h + |A|^2 h$ and $D_{\pdv{t}} H = \Delta H + |A|^2 H$.
It follows that 
\begin{align*}
\mathcal Y(\varphi) - \mathcal X_i(\mathcal X_i(\varphi)) &=  (D_{\pdv{t}}h - \Delta h)(e_1, e_1) + (D_{\pdv{t}}h - \Delta h)(e_2, e_2) - \beta(\pdv{t} H - \Delta H)\\
& = |A|^2 \varphi \geq 0. 
\end{align*}
This is a degenerate elliptic equation for $\varphi$. Let $\mathcal F = \{\underline e \in \mathcal O : \varphi(\underline e ) = 0\}$ denote the zero set of $\varphi$. Fix a time $\tau \in (-T, 0)$. We claim the set of all two-frames $\{e_1,e_2\}$ that are orthonormal with respect to $g(\tau)$ and satisfy $h_{g(\tau)}(e_1, e_1) + h_{g(\tau)}(e_2, e_2) = \beta H_{g(\tau)}$ is invariant under parallel transport. Let $\gamma : [0, 1] \to \Omega$ be a smooth path and let $\underline v(s) := \{v_1(s), \dots, v_n(s)\}$ be a parallel orthonormal frame along $\gamma$ for $s \in [0, 1]$ with respect to $g(\tau)$. This defines a smooth path $\underline v : [0, 1] \to \mathcal O$ with the property that $\underline v(s)$ lies above $\gamma(s)$ and $\underline v'(s)$ is the horizontal lift of $\gamma'(s)$. Evidently, we can find smooth functions $f_1, \dots, f_n : [0, 1] \to \R$ such that $\gamma'(s) = \sum_{j =1}^n f_j(s) v_j(s)$ for $s \in [0, 1]$ and this implies that 
\[
\underline v'(s) = \sum_{j =1}^n f_j(s) \mathcal X_j|_{\underline v(s)}. 
\]
If we assume that $h_{g(\tau)}(v_1(0), v_1(0)) + h_{g(\tau)}(v_2(0), v_2(0)) = \beta H_{g(\tau)}$, then $\underline v(0) \in \mathcal F$. At this point, all of the assumptions of Bony's strict maximum principle for degenerate elliptic equations (Corollary 9.7 in \cite{Bre10}) have been verified, and thus we conclude $\underline v(s) \in \mathcal F$ for $s \in [0, 1]$. This completes the proof.

\sc{Department of Mathematics, Columbia University, New York, NY 10027}

\end{document}